\documentclass[12pt]{amsart}

\usepackage{framed}
\usepackage{braket}
\usepackage{amsmath,amssymb,amsthm}
\usepackage{color}
\usepackage{bm}
\usepackage[all]{xy}
\usepackage{amscd}
\usepackage{graphicx}
\usepackage{ascmac}
\usepackage{BOONDOX-frak, mathrsfs}
\usepackage[top=30truemm,bottom=30truemm,left=25truemm,right=25truemm]{geometry}
\usepackage{mathtools}
\mathtoolsset{showonlyrefs=true}

\makeatletter
\@addtoreset{equation}{section}

\makeatother

\newcommand{\Z}{\mathbb{Z}}
\newcommand{\Q}{\mathbb{Q}}

\newcommand{\sN}{\mathsf{N}}
\newcommand{\pe}{\mathfrak{p}}
\newcommand{\fp}{\mathfrak{p}}

\newcommand{\qu}{\mathfrak{q}}
\newcommand{\Qu}{\mathfrak{Q}}

\newcommand{\ttilde}{\widetilde}
\newcommand{\dual}{\vee}

\newcommand{\PP}{\mathcal{P}}

\newcommand{\CC}{\mathcal{C}}
\newcommand{\RR}{\mathcal{R}}

\newcommand{\cG}{\mathcal{G}}

\newcommand{\UU}{\mathcal{U}}

\newcommand{\TT}{\mathcal{T}}

\newcommand{\cA}{\mathcal{A}}
\newcommand{\BB}{\mathcal{B}}
\newcommand{\LL}{\mathcal{L}}
\newcommand{\cS}{\mathcal{S}}

\newcommand{\VV}{\mathcal{V}}

\newcommand{\XX}{X}

\newcommand{\De}{D}

\newcommand{\DeP}{\De^{\perf}}

\newcommand{\otimesL}{\otimes^{\mathbb{L}}}

\newcommand{\derR}{\mathsf{R}}
\newcommand{\RG}{\derR\Gamma}
\newcommand{\RHom}{\derR\Hom}

\newcommand{\ol}[1]{\overline{#1}}

\newcommand{\iDet}{\mathrm{d}}

\DeclareMathOperator{\Gal}{Gal}

\DeclareMathOperator{\Ker}{Ker}
\DeclareMathOperator{\Coker}{Coker}
\DeclareMathOperator{\Image}{Im}

\DeclareMathOperator{\perf}{perf}
\DeclareMathOperator{\pd}{pd}
\DeclareMathOperator{\Fitt}{Fitt}
\DeclareMathOperator{\Det}{Det}
\DeclareMathOperator{\rank}{rank}
\DeclareMathOperator{\Hom}{Hom}

\DeclareMathOperator{\tor}{tor}
\DeclareMathOperator{\ram}{ram}

\DeclareMathOperator{\fin}{fin}

\DeclareMathOperator{\height}{ht}
\DeclareMathOperator{\alg}{alg}
\DeclareMathOperator{\Iw}{Iw}
\DeclareMathOperator{\Ann}{Ann}

\DeclareMathOperator{\id}{id}
\DeclareMathOperator{\loc}{loc}
\DeclareMathOperator{\Ext}{Ext}
\DeclareMathOperator{\ddet}{det}
\DeclareMathOperator{\cha}{char}
\DeclareMathOperator{\length}{length}
\DeclareMathOperator{\PN}{PN}
\DeclareMathOperator{\cd}{cd}

\let\oldenumerate\enumerate
\renewcommand{\enumerate}{
   \oldenumerate
   \setlength{\itemsep}{1pt}
   \setlength{\parskip}{0pt}
   \setlength{\parsep}{0pt}
}
\let\olditemize\itemize
\renewcommand{\itemize}{
   \olditemize
   \setlength{\itemsep}{1pt}
   \setlength{\parskip}{0pt}
   \setlength{\parsep}{0pt}
}

\theoremstyle{plain}
\newtheorem{thm}{Theorem}[section]
\newtheorem{lem}[thm]{Lemma}

\newtheorem{prop}[thm]{Proposition}
\newtheorem{cor}[thm]{Corollary}

\newtheorem{ass}[thm]{Assumption}
\theoremstyle{definition}
\newtheorem{defn}[thm]{Definition}

\newtheorem{rem}[thm]{Remark}

\title[Higher codimension Iwasawa theory]
{Higher codimension behavior in equivariant Iwasawa theory for CM-fields}
\author[T. Kataoka]{Takenori Kataoka}
\address{Faculty of Science and Technology, Keio University.
3-14-1 Hiyoshi, Kohoku-ku, Yokohama, Kanagawa 223-8522, Japan}
\email{tkataoka@math.keio.ac.jp}
\date{}

\begin{document}

\begin{abstract}
In classical Iwasawa theory, we mainly study codimension one behavior of arithmetic modules.
Relatively recently, F.~M.~Bleher, T.~Chinburg, R.~Greenberg, M.~Kakde, G.~Pappas, R.~Sharifi, and M.~J.~Taylor started studying higher codimension behavior of unramified Iwasawa modules which are conjectured to be pseudo-null.
In this paper, by developing a general algebraic theory on perfect complexes, we obtain a new perspective of their work.
That allows us to extend the results to equivariant settings and, even in non-equivariant settings, to obtain more refined results concerning the higher codimension behavior.
\end{abstract}

\maketitle

\tableofcontents

\section{Introduction}\label{sec:01}

In classical Iwasawa theory we usually study various modules up to pseudo-null modules.
For instance, Iwasawa main conjecture is often formulated as a relation between the characteristic ideals of Iwasawa modules and $p$-adic $L$-functions, and in general characteristic ideals ignore pseudo-null modules.
On the other hand, Greenberg's conjecture claims that the unramified Iwasawa modules are pseudo-null in certain situations (see \cite[Conjecture 3.4.1]{BCG+}).
In that case, therefore, there was not satisfactory research on more detailed structure of unramified Iwasawa modules.

Against this background, F.~M.~Bleher, T.~Chinburg, R.~Greenberg, M.~Kakde, G.~Pappas, R.~Sharifi, and M.~J.~Taylor \cite{BCG+} began studying unramified Iwasawa modules that are assumed to be pseudo-null.
More concretely, when the base field is an imaginary quadratic field, 
they obtained a relation between unramified Iwasawa modules and pairs of $p$-adic $L$-functions.
In subsequent work \cite{BCG+b},
when the base field is a CM-field, an analogous result is obtained, though we have to replace unramified Iwasawa modules by certain alternatives that are defined via exterior powers.
An analogue for Selmer groups of elliptic curves is also studied by A.~Lei and B.~Palvannan \cite{LP19}.

In this paper, we extend and refine the results of \cite{BCG+} and \cite{BCG+b} from (abelian) equivariant perspective.
The term ``equivariant'' means that we allow $p$ to divide the order of the finite abelian extension concerned.
In that case, the ring theoretic property of the Iwasawa algebra gets worse; for instance, it is no longer normal, so the characteristic ideals are not defined in a classical way.
We obtain equivariant versions of the results of \cite{BCG+} and \cite{BCG+b} by studying arithmetic complexes whose cohomology groups know the Iwasawa modules.
Our method also gives refined results even in the non-equivariant setting.
For example, a main result of \cite{BCG+b} (which we recall in Theorem \ref{thm:81}) is formulated using the second Chern classes of modules, so it studies the codimension two behavior only.
We generalize this theorem in a form without localization at height two prime ideals, so it also knows higher codimension behavior.

This paper has two kinds of main results, which respectively generalize main results of \cite{BCG+b} and \cite{BCG+}.
In \S \ref{subsec:24} and \S \ref{subsec:15}, we illustrate those main results.
In \S \ref{subsec:49}, we explain the central idea of this paper to prove them.

\subsection{The first main theorem}\label{subsec:24}

We state the first main theorem in a special case; the general statement is Theorem \ref{thm:103}.
In order not to impair the readability, for now we omit to introduce all the notations that are necessary for the precise statement and instead refer to later sections for them.

Let $p$ be a fixed odd prime number.
Let $E$ be a CM-field.
We write $S_p(E)$ (resp.~$S_{\infty}(E)$) for the set of $p$-adic primes (resp.~infinite places) of $E$.
A subset $\cS \subset S_p(E)$ is called a ($p$-adic) CM-type if $S_p(E)$ is the disjoint union of $\cS$ and $\ol{\cS}$, where $\ol{\cS}$ denotes the set of the complex conjugates of primes in $\cS$.
We assume that there exists a CM-type for $E$, namely, that $E$ satisfies the $p$-ordinary condition.

We write $\tilde{E}$ for the compositum of all $\Z_p$-extensions of $E$.
We consider an abelian extension $K/E$ such that $K$ contains $\tilde{E}(\mu_p)$ and $K/\tilde{E}$ is a finite extension.
Here, $\mu_p$ denotes the group of $p$-th roots of unity.
A key point is that we deal with the ``equivariant'' case, namely, we allow the $p \mid [K: \tilde{E}]$ case.
The articles \cite{BCG+} and \cite{BCG+b} deal with the $p \nmid [K: \tilde{E}]$ case only.

For a finite set $S$ of finite primes of $E$, let $\XX_{S}(K)$ denote the $S$-ramified Iwasawa module for $K$, which is by definition the Galois group of the maximal abelian $p$-extension of $K$ that is unramified at all primes not lying above $S$.
We also put $X(K) = X_{\emptyset}(K)$.
It is known that $X_S(K)$ is a finitely generated module over the Iwasawa algebra $\RR = \Z_p[[\Gal(K/E)]]$.
It is expected that $X_{\cS}(K)$ is a torsion $\RR$-module (i.e., annihilated by a non-zero-divisor) if $\cS$ is a CM-type (see Assumption \ref{ass:41} and Remark \ref{rem:92}).
In this introduction, we assume this property.

Let $\Sigma$ be a finite set of places of $E$ that contains $S_p(E) \cup S_{\infty}(E)$.
We also impose an additional condition labeled as \eqref{eq:60}.
Note that if $p \nmid [K: \tilde{E}]$, this condition is trivial so we may take $\Sigma = S_p(E) \cup S_{\infty}(E)$.
We write $\Sigma_f$ for the set of finite primes in $\Sigma$.

Let $\cS$ be a CM-type.
Then we have a natural exact sequence of $\RR$-modules
\[
D_{\cS}(K) \oplus D_{\ol{\cS}}(K) \to \XX_{\Sigma_f}(K) \to \XX_{\Sigma_f \setminus S_p(E)}(K) \to 0.
\]
Here, the modules $D_{\cS}(K)$ and $D_{\ol{\cS}}(K)$ are defined using local information (see \S \ref{subsec:60}).
This sequence implies that the first map offers information on the Iwasawa module $\XX_{\Sigma_f \setminus S_p(E)}(K)$.
We can check that (Lemma \ref{lem:rank}) the generic ranks over $\RR$ of $D_{\cS}(K)$, $D_{\ol{\cS}}(K)$, and $\XX_{\Sigma_f}(K)$ are all $d = [E: \Q]/2$.
Then, inspired by \cite{BCG+b}, 
we take the $d$-th exterior powers, which gives rise to a map
\[
\bigwedge_{\RR}^d D_{\cS}(K) \oplus \bigwedge_{\RR}^d D_{\ol{\cS}}(K) 
\to \bigwedge_{\RR}^d \XX_{\Sigma_f}(K).
\]
The first main theorem studies the cokernel of this map or, more accurately, the cokernel after taking the quotient of the target module modulo its torsion part, which is denoted by $(-)_{/\tor}$.
We will introduce the other notations used in the statement afterward.

\begin{thm}\label{thm:102}
We have an exact sequence
\[
0 \to \left(\frac{\left(\bigwedge_{\RR}^d \XX_{\Sigma_f}(K) \right)_{/\tor}}
{\bigwedge_{\RR}^d D_{\cS}(K) + \bigwedge_{\RR}^d D_{\ol{\cS}}(K) }\right)_{\qu}
\to \frac{\RR_{\qu}}{(\LL^{\alg}_{\Sigma, \Sigma_f \setminus \cS}, \LL^{\alg}_{\Sigma, \Sigma_f \setminus \ol{\cS}})}
\to \frac{\RR_{\qu}}{\Fitt_{\RR_{\qu}}(H^2_{\Iw}(K_{\Sigma}/K, \Z_p)^{\iota}_{\qu})}
\to 0
\]
for each prime ideal $\qu$ of $\Lambda$ which is not in the support of 
\[
\bigoplus_{\pe \in S_p(E)} Z_{\pe}(K) \oplus \bigoplus_{\pe \in S_p(E)} Z_{\pe}(K)^{\iota}(1).
\]
\end{thm}

We briefly explain the notations (see also \S \ref{subsec:130}):

\begin{itemize}
\item
$\Lambda$ is an auxiliary subalgebra of $\RR$ (see \S \ref{subsec:131}), which is introduced only to suitably formulate ``avoiding the supports of $Z_{\fp}(K)$ and $Z_{\fp}(K)^{\iota}(1)$.''

\item
$Z_{\fp}(K)$ is an $\RR$-module which is isomorphic to $\Z_p[[\Gal(K/E)/G_{\fp}(K/E)]]$, where $G_{\fp}(K/E)$ denotes the decomposition group (see \S \ref{subsec:60}).

\item
$H^2_{\Iw}(K_{\Sigma}/K, \Z_p)$ is the Iwasawa cohomology group, which is closely related to $X(K)$ (see \S \ref{subsec:60}).

\item
For a subset $S \subset \Sigma_f$ such that $S \cap S_p(E)$ is a CM-type, $\LL_{\Sigma, S}^{\alg}$ denotes the algebraic $p$-adic $L$-function (see \S \ref{sec:23}).
Note that in this paper we do not study a main conjecture (i.e., a relation with analytic aspects), but see the final paragraph of \S \ref{sec:23}.
\end{itemize}

In \S \ref{subsec:132}, we will explain how to recover a main theorem of \cite{BCG+b} from our theorem.
A main tool is the additivity of the (second) Chern classes with respect to exact sequences, though we will also need an additional algebraic proposition proved in \S \ref{sec:87}.
Let us also stress that our main theorem does not require localization at height two prime ideals, which is an improvement on \cite{BCG+b}.

It is natural to ask if the exterior powers in Theorem \ref{thm:102} have arithmetic interpretations.
In \cite[Theorem D]{BCG+b}, an answer is given when $d = 2$ (more generally when $n = 2$ and $l = 2$ in the notation of Theorem \ref{thm:103}).
This answer is actually also valid in our equivariant situation without changing the discussion, so we omit the details.

\subsection{The second main theorem}\label{subsec:15}

Contrary to the first main theorem, our second main theorem (Theorem \ref{thm:105}) describes, under more restrictions (including $d = 1$), unramified Iwasawa modules themselves without avoiding any prime ideal $\qu$.

\begin{thm}\label{thm:101}
Suppose that $E$ is an imaginary quadratic field (i.e., $d = 1$).
Let $S_p(E) = \{\pe, \ol{\pe}\}$.
Suppose that we may take $\Sigma = S_p(E) \cup S_{\infty}(E)$ (note that this is unfortunately restrictive; see condition \eqref{eq:60}).
Also, suppose that $\XX(K)$ is pseudo-null.
Then we have an exact sequence of $\RR$-modules
\[
\Z_p(1) \to \XX(K) \to \frac{\RR}{(\LL_{\Sigma, \{\pe\}}^{\alg}, \LL_{\Sigma, \{\ol{\pe}\}}^{\alg})} \to E^2(\XX(K))^{\iota}(1) \to 0.
\]
Moreover, the image of the first map from $\Z_p(1)$ is finite.
\end{thm}

Here, we put $E^2(-) = \Ext^2_{\RR}(-, \RR)$ (see \S \ref{subsec:130} for the convention on $\RR$-module structure).

When $p \nmid [K: \tilde{E}]$, the character-component of Theorem \ref{thm:101} recovers \cite[Theorem 5.2.1]{BCG+}.
We also have the following corollary of Theorem \ref{thm:101} (see Corollary \ref{cor:73} for a generalization).
Let $\XX(K)_{\fin}$ denote the maximal finite submodule of $\XX(K)$.

\begin{cor}\label{cor:72}
In the situation of Theorem \ref{thm:101}, 
$\XX(K)_{\fin}$ is a quotient of $\Z_p(1)$ as an $\RR$-module.
\end{cor}

\subsection{The key idea of the proof}\label{subsec:49}

We outline the proofs of the main theorems, putting focus on difference from \cite{BCG+}, \cite{BCG+b}.

\subsubsection{Review of methods of \cite{BCG+} and \cite{BCG+b}}

First we explain a key idea of \cite{BCG+} and \cite{BCG+b}.
We assume $p \nmid [K: \tilde{E}]$ and $\Sigma_f = S_p(E)$.
In this case, $\RR$ is a finite product of regular local rings.

The following algebraic proposition is a key observation.
Recall that a finitely generated module $M$ over a regular local ring is said to be reflexive if the canonical homomorphism $\alpha_M: M \to M^{**}$ is isomorphic, where in general $(-)^*$ denotes the linear dual.
It is known that $M^{**}$ is automatically reflexive, and $M^{**}$ is called the reflexive hull of $M$.

\begin{prop}\label{prop:51}
Let $A$ be a regular local ring and $M$ a reflexive $A$-module.
\begin{itemize}
\item[(1)] If $\dim(A) \leq 2$, then $M$ is free.
\item[(2)] (\cite[Lemma A.1]{BCG+}). If the generic rank of $M$ over $A$ is one, then $M$ is free.
\end{itemize}
\end{prop}

Then a key idea of \cite[\S 5]{BCG+} and \cite[\S 4]{BCG+b} is to replace the arithmetic modules $\XX_{S_p(E)}(K)$ and $D_{\cS}(K)$ by their reflexive hulls.
We consider a natural commutative diagram
\begin{equation}\label{eq:55}
\xymatrix{
	\bigwedge_{\RR}^d D_{\cS}(K) \oplus \bigwedge_{\RR}^d D_{\ol{\cS}}(K) \ar[r] \ar[d]
	& \bigwedge_{\RR}^d D_{\cS}(K)^{**} \oplus \bigwedge_{\RR}^d D_{\ol{\cS}}(K)^{**} \ar[d]\\
	\bigwedge_{\RR}^d \XX_{S_p(E)}(K) \ar[r]
	& \bigwedge_{\RR}^d \XX_{S_p(E)}(K)^{**}.
}
\end{equation}

Concerning Theorem \ref{thm:102}:
After taking $(-)_{/\tor}$ of the lower left module, the cokernel of the left vertical arrow is what we want to understand.
Thanks to Proposition \ref{prop:51}(1), the modules on the right are free after localization at primes of height two.
Hence the right vertical arrow is easy to understand, yielding the middle term of the sequence in Theorem \ref{thm:102} that involves the algebraic $p$-adic $L$-functions.

Another important point is that the cokernel of the maps to the reflexive hulls can be explicitly described via spectral sequences from duality theorems (Proposition \ref{prop:20}).
Then we can also investigate the cokernels of the horizontal arrows in the above diagram.
By the assumption on $\qu$, we can see that the cokernel of the upper arrow vanishes, and that the cokernel of the lower arrow yields the last term of the sequence in Theorem \ref{thm:102} that involves the Iwasawa cohomology group.
Using these facts, we obtain the result by applying snake lemma.

Concerning Theorem \ref{thm:101} (so $d = 1$):
The idea is the same, but we can remove several defects in the statement of Theorem \ref{thm:102}, as follows.
By Proposition \ref{prop:51}(2), we do not have to take localization at height two primes.
It is known that $\XX_{S_p(E)}(K)$ is torsion-free (under our hypothesis that $\XX(K)$ is pseudo-null), so we do not have to take $(-)_{/\tor}$.
Finally, the vertical arrow between the cokernels of the horizontal arrows can be studied directly, so we do not have to avoid $\qu$ as in Theorem \ref{thm:102}.
These facts imply the result.

\subsubsection{Idea of this paper}

A difficulty in generalization to the $p \mid [K: \tilde{E}]$ case is that, contrary to Proposition \ref{prop:51}, reflexive modules cannot be expected to be free over our algebra $\RR$.

Our key idea in this paper is to use, instead of exterior powers of reflexive hulls, determinant modules of perfect complexes, which are free of rank one by definition.
More concretely, in Proposition \ref{prop:30}, we introduce a natural homomorphism
\[
\Psi_C: \bigwedge_{\RR}^d H^1(C) \to \Det_{\RR}^{-1}(C)
\]
for a perfect complex $C$ over $\RR$ that satisfies several conditions.
The kernel of $\Psi_C$ is precisely the torsion part, and a description of the cokernel of $\Psi_C$ is also obtained.
Note that, contrary to the reflexive hull, the cokernel of $\Psi_C$ is not necessarily pseudo-null, but this does not matter for our purpose.

The arithmetic modules $\XX_{\Sigma_f}(K)$ and $D_{\cS}(K)$ have interpretations as the first cohomology of perfect complexes $C_{\Sigma, \Sigma_f}[1]$ and $C_{\cS}^{\loc}$ respectively (\S \ref{subsec:60}), and we can apply the general construction of the maps to the determinant modules.
As a consequence, we obtain a commutative diagram
\begin{equation}\label{eq:56}
\xymatrix{
	\bigwedge_{\RR}^d D_{\cS}(K) \oplus \bigwedge_{\RR}^d D_{\ol{\cS}}(K) \ar[r] \ar[d]
	& \Det_{\RR}^{-1}(C_{\cS}^{\loc}) \oplus \Det_{\RR}^{-1}(C_{\ol{\cS}}^{\loc}) \ar[d]\\
	\bigwedge_{\RR}^d \XX_{\Sigma_f}(K) \ar[r]
	& \Det_{\RR}(C_{\Sigma, \Sigma_f})
}
\end{equation}
(see Proposition \ref{prop:46}).
This diagram plays the same role as \eqref{eq:55}.
The cokernel of the right vertical arrow can be described by the algebraic $p$-adic $L$-functions (this is almost the definition), yielding the middle term of Theorem \ref{thm:102}.
The cokernel of the horizontal arrows can be described by using duality theorems.
By these facts, we obtain Theorem \ref{thm:102} again by the snake lemma.

The proof of Theorem \ref{thm:101}, however, requires reflexive hulls.
A main contribution of this paper is, under the assumption of Theorem \ref{thm:101}, to prove that the reflexive hulls of $\XX_{S_p(E)}(K)$ and $D_{\cS}(K)$ are free of rank one.
In fact, we will show that the reflexive hulls are isomorphic to the determinant modules that appeared in \eqref{eq:56} (recall $d = 1$).
Then the proof proceeds similarly as in the $p \nmid [K: \tilde{E}]$ case.

\subsection{Organization of this paper}\label{subsec:141}

In \S \ref{sec:02}, we will establish the key algebraic tool for the main theorems, as explained in \S \ref{subsec:49}.
In \S \ref{sec:14}, we review several facts on arithmetic complexes, including duality theorems.
In \S \ref{sec:23}, we introduce the algebraic $p$-adic $L$-functions, which are used in the statements of the main theorems.
In \S \S \ref{sec:33}, \ref{sec:11}, we prove the two main theorems.
\S \ref{sec:87} is an appendix on some properties of Fitting ideals that are used in \S \ref{subsec:132}.

\subsection{A list of notations}\label{subsec:130}

For a field $k$, let $\ol{k}$ denote an algebraic closure of $k$.

The superscript $(-)^{\dual}$ denotes the Pontryagin dual for modules which are either compact or discrete.

For a (commutative) ring $A$, let $\dim(A)$ denote the Krull dimension of $A$.
Let $Q(A)$ denote the total ring of fractions of $A$.
For a prime ideal $\fp$ of $A$, let $\height(\fp)$ denote its height.

Let $M$ be a finitely generated module over a noetherian ring $A$.
For an integer $l \geq 0$, let $\bigwedge_{A}^l M$ denote the $l$-th exterior power of $M$.
Let $\pd_A(M)$ denote the projective dimension of $M$.
Let $\Ann_A(M)$ denote the annihilator ideal of $M$.

Let $\Fitt_A(M)$ denote the (initial) Fitting ideal of $M$.
By definition, if $A^a \overset{H}{\to} A^b \to M \to 0$ is a finite presentation of $M$, then $\Fitt_A(M)$ is generated by all $b \times b$ minors of the presentation matrix of $H$.

We put $E^i(M) = \Ext^i_{A}(M, A)$ for integers $i \geq 0$.
In particular, we define the linear dual by $M^* = E^0(M) = \Hom_A(M, A)$.
Then $E^i(M)$ is an $A$-module; when $i = 0$, we define the module structure by $(a\phi)(x) = \phi(a x)$ for $a \in A$, $\phi \in \Hom_{A}(M, A)$, and $x \in M$.
Note that this is the opposite of the convention adopted in \cite{BCG+} and \cite{BCG+b}.

Let $\cG$ be an abelian compact group which contains an open subgroup $\Gamma$ which is (non-canonically) isomorphic to $\Z_p^d$ for some $d \geq 1$.
Let $M$ be a finitely generated module over $\Z_p[[\cG]]$.
We say that $M$ is torsion over $\Z_p[[\cG]]$ if it is annihilated by a non-zero-divisor of $\Z_p[[\cG]]$.
We also say that $M$ is pseudo-null if $M_{\Qu} = 0$ for any prime $\Qu$ of $\Z_p[[\cG]]$ with $\height(\Qu) \leq 1$.
Then being torsion (resp.~pseudo-null) over $\Z_p[[\cG]]$ is equivalent to being torsion (resp.~pseudo-null) over $\Z_p[[\Gamma]]$.
Therefore, we have no afraid of confusion about the base ring for these properties.
Let $M_{\tor}$ (resp.~$M_{\PN}$, resp.~$M_{\fin}$) denote the maximal torsion (resp.~pseudo-null, resp.~finite) submodule of $M$.
We also put $M_{/\bullet} = M/M_{\bullet}$ for $\bullet \in \{\tor, \PN, \fin\}$.

\section{The key algebraic propositions}\label{sec:02}

As discussed in \S \ref{subsec:49}, this section constitutes the technical heart of this paper.

\subsection{Preliminaries on perfect complexes}\label{subsec:a21}

For a noetherian ring $A$, let $\DeP(A)$ be the derived category of perfect complexes over $A$.
For integers $a \leq b$, let $\De^{[a, b]}(A)$ be the full subcategory of $\DeP(A)$ that consists of complexes 
which are quasi-isomorphic to complexes of the form
\[
[\cdots \to 0 \to C^a \to \dots \to C^b \to 0 \to \cdots],
\]
concentrated in degrees $a, a+1, \dots, b$, where $C^i$ is finitely generated and projective for each $a \leq i \leq b$.
For each $C \in \DeP(A)$ which is quasi-isomorphic to the displayed complex, we define the Euler characteristic of $C$ by
\[
\chi_{A}(C) = \sum_i (-1)^{i-1} \rank_A(C^i).
\]
Here, $\rank_A(-)$ denotes the (locally constant) rank for projective modules.
The determinant module of $C$ is defined by
\[
\Det_{A}(C) = \bigotimes_i \ddet_{A}^{(-1)^i}(C^i),
\]
where, for a finitely generated projective $A$-module $F$, we put
\[
\ddet_{A}(F) = \bigwedge_{A}^{\rank_A(F)} F,
\qquad \ddet_{A}^{-1}(F) = \ddet_{A}(F)^*
\]
(recall that $(-)^*$ denotes the linear dual).
We also put
\[
\Det_{A}^{-1}(C) =  \Det_{A}(C)^*.
\]
Since no problem occurs, throughout this paper, we ignore the degrees on determinant modules.

Given an $A$-algebra $A'$, for $C \in D^{[a, b]}(A)$, let $A' \otimesL_A C \in D^{[a, b]}(A')$ denote the derived tensor product.
Then we have a natural isomorphism $A' \otimes_A \Det_A(C) \simeq \Det_{A'}(A' \otimesL_A C)$.

\subsection{The key propositions}\label{subsec:a22}

Let $\RR$ be a (commutative) ring which contains a regular local ring $\Lambda$ such that $\RR$ is free of finite rank over $\Lambda$ and moreover there exists an isomorphism
\[
\Hom_{\Lambda}(\RR, \Lambda) \simeq \RR
\]
as $\RR$-modules.
Note that this implies that $\RR$ is Gorenstein.
Moreover, we obtain isomorphisms $\Ext_{\RR}^i(M, \RR) \simeq \Ext_{\Lambda}^i(M, \Lambda)$ for $\RR$-modules $M$, so no confusion about the base ring occurs when we write $E^i(M)$.

The following is a special case of Proposition \ref{prop:30}.

\begin{prop}\label{prop:31}
Let $C \in \De^{[0, 1]}(\RR)$ with $H^0(C) = 0$ and set $l = \chi_{\RR}(C)$.
Then we have a natural homomorphism
\[
\Psi_C: \bigwedge_{\RR}^l H^1(C) \to \Det_{\RR}^{-1}(C)
\]
which satisfies the following.
\begin{itemize}
\item[(1)]
We have $\Ker(\Psi_C) = \left( \bigwedge_{\RR}^l H^1(C)\right)_{\tor}$.
\item[(2)]
We have an isomorphism
\[
\Coker(\Psi_C) \simeq \RR / \Fitt_{\RR}(E^1(H^1(C))).
\]
\end{itemize}
\end{prop}

\begin{proof}
We consider a natural homomorphism
\[
\ol{\Psi_C}: \bigwedge_{\RR}^l H^1(C) 
\to Q(\RR) \otimes_{\RR} \bigwedge_{\RR}^l H^1(C) 
\simeq \Det_{Q(\RR)}^{-1}(Q(\RR) \otimesL_{\RR} C)
\simeq Q(\RR) \otimes_{\RR} \Det_{\RR}^{-1}(C),
\]
where the middle isomorphism follows from $Q(\RR) \otimesL_{\RR} C \in \De^{[1, 1]}(Q(\RR))$.
We shall show that $\Image(\ol{\Psi_C}) \subset \Det_{\RR}^{-1}(C)$, and define $\Psi_C$ as the induced homomorphism.
In fact, we will give a more explicit description of $\Psi_C$.

Let us take a quasi-isomorphism $C \simeq [C^0 \overset{\alpha}{\to} C^1]$ with $C^0$ and  $C^1$ finitely generated and projective.
Put $a = \rank_{\RR}(C^0)$, so $\rank_{\RR}(C^1) = a+l$.
We consider the exact sequence
\[
0 \to C^0 \overset{\alpha}{\to} C^1 \overset{\beta}{\to} H^1(C) \to 0.
\]
Then we are able to define a homomorphism
\[
\Psi_C': \bigwedge_{\RR}^{a} C^0 \otimes_{\RR} \bigwedge_{\RR}^l H^1(C) 
\to \bigwedge_{\RR}^{a + l} C^1
\]
by
\[
(x_1 \wedge \dots \wedge x_a) \otimes (\beta(y_1) \wedge \dots \wedge \beta(y_l)) 
\mapsto \alpha(x_1) \wedge \dots \wedge \alpha(x_a) \wedge y_1 \wedge \dots \wedge y_l
\]
for $x_1, \dots, x_a \in C^0$ and $y_1, \dots, y_l \in C^1$.
The well-definedness of $\Psi_C'$ follows from $\bigwedge_{\RR}^{a + 1} C^0 = 0$ as in \cite[Lemma 2.1]{Sak18}.
Taking $\Det_{\RR}^{-1}(C^0) \otimes_{\RR} (-)$ to the both sides, we obtain a homomorphism
\[
\bigwedge_{\RR}^l H^1(C) 
\to \Hom_{\RR} \left( \bigwedge_{\RR}^{a} C^0, \bigwedge_{\RR}^{a + l} C^1 \right) \simeq \Det_{\RR}^{-1}(C).
\]
By construction, the homomorphism $\ol{\Psi_C}$ coincides with this map (followed by the canonical inclusion).
This implies $\Image(\ol{\Psi_C}) \subset \Det_{\RR}^{-1}(C)$ as claimed.

Then assertion (1) is clear from the construction of $\Psi_C$.
For assertion (2), we first note an isomorphism
\[
\Coker(\Psi_C)
\simeq \Coker(\Psi_C').
\]
Then we consider maps
\[
\Det_{\RR}(C^0) \otimes_{\RR} \bigwedge_{\RR}^l C^1
\overset{\id \otimes \bigwedge^l \beta}{\twoheadrightarrow} \Det_{\RR}(C^0) \otimes_{\RR} \bigwedge_{\RR}^l H^1(C)
\overset{\Psi_C'}{\to} \Det_{\RR}(C^1).
\]
By a direct computation after choosing bases of $C^0$ and $C^1$, we see that $\Image(\Psi_C')$ 
in $\Det_{\RR}(C^1) \simeq \RR$ is generated by the $a \times a$ minors of the matrix $\alpha$.
By the induced exact sequence
\[
0 \to H^1(C)^* \to (C^1)^* \overset{\alpha^*}{\to} (C^0)^* \to E^1(H^1(C)) \to 0,
\]
and by the definition of Fitting ideals, we obtain 
\[
\Coker(\Psi_C') \simeq \RR / \Fitt_{\RR}(E^1(H^1(C))).
\]
This completes the proof.
\end{proof}

For a prime $\qu$ of $\Lambda$, the subscript $(-)_{\qu}$ denotes the localization with respect to the multiplicative set $\Lambda \setminus \qu$.

The following is the key proposition, which we prove by using Proposition \ref{prop:31}.

\begin{prop}\label{prop:30}
Let $C$ be a perfect complex over $\RR$ such that $H^i(C)$ is pseudo-null for any $i \neq 1$.
We set $l = \chi_{\RR}(C)$.
Then we have a natural homomorphism
\[
\Psi_C: \bigwedge_{\RR}^l H^1(C) \to \Det_{\RR}^{-1}(C)
\]
which satisfies the following.
\begin{itemize}
\item[(1)]
We have $\Ker(\Psi_C) = \left( \bigwedge_{\RR}^l H^1(C)\right)_{\tor}$.
\item[(2)]
Let $\qu$ be a prime ideal of $\Lambda$ such that $\RR_{\qu} \otimesL_{\RR} C \in \De^{[0, 1]}(\RR_{\qu})$ and that $H^0(C)_{\qu} = 0$.
Then we have an isomorphism
\[
\Coker(\Psi_C)_{\qu} \simeq \RR_{\qu} / \Fitt_{\RR_{\qu}}(E^1(H^1(C))_{\qu}).
\]
\item[(3)]
Let $\qu$ be a prime ideal of $\Lambda$ such that $\RR_{\qu} \otimesL_{\RR} C \in \De^{[1, 2]}(\RR_{\qu})$ and that $\pd_{\RR_{\qu}}(H^2(C)_{\qu}) \leq 2$.
Then $\left(\Psi_C\right)_{\qu}$ is bijective.
\end{itemize}
\end{prop}

\begin{proof}
Since $H^i(C)$ is torsion for any $i \neq 1$, as in the proof of Proposition \ref{prop:31}, we have a natural homomorphism
\[
\ol{\Psi_C}: \bigwedge_{\RR}^l H^1(C) 
\to 
Q(\RR) \otimes_{\RR} \Det_{\RR}^{-1}(C).
\]
We claim that $\Image(\ol{\Psi_C}) \subset \Det_{\RR}^{-1}(C)$.
Let $\qu$ be any prime ideal of $\Lambda$ with $\height(\qu) \leq 1$.
By the assumption, we have $H^i(C)_{\qu} = 0$ for $i \neq 1$.
Hence the complex $\RR_{\qu} \otimesL_{\RR} C$ satisfies the condition in Proposition \ref{prop:31} over $\RR_{\qu}$ (since $\dim(\RR_{\qu}) \leq 1$).
Therefore, Proposition \ref{prop:31} implies that $\Image(\ol{\Psi_C})_{\qu} \subset \Det_{\RR}^{-1}(C)_{\qu}$.
Since $\Det_{\RR}^{-1}(C)$ is reflexive, this proves $\Image(\ol{\Psi_C}) \subset \Det_{\RR}^{-1}(C)$.

Therefore, we are able to define $\Psi_C$ as the homomorphism induced by $\ol{\Psi_C}$.
Now assertion (1) is clear and assertion (2) follows directly from Proposition \ref{prop:31}(2).
We show assertion (3).
By the assumptions, the module $H^1(C)_{\qu}$ is projective over $\RR_{\qu}$ of rank $l$.
Therefore, $\left(\Psi_C\right)_{\qu}$ is a homomorphism between free modules of rank one.
For any height one prime $\qu'$ of $\Lambda$ contained in $\qu$, since $\RR_{\qu'} \otimesL_{\RR} C \in \De^{[1, 1]}(\RR_{\qu'})$ and $H^1(C)_{\qu'}$ is projective, assertion (2) implies that $\left(\Psi_C\right)_{\qu'}$ is surjective.
Hence we see that $\left(\Psi_C \right)_{\qu}$ is bijective.
\end{proof}

\subsection{Relations with reflexive hulls}\label{subsec:a23}

Let $\RR \supset \Lambda$ be as in \S \ref{subsec:a22}.
Recall that, for a finitely generated $\RR$-module $M$, the reflexive hull is defined as $M^{**}$ together with the natural map $\alpha_M: M \to M^{**}$.
It is known that the kernel of $\alpha_M$ is $M_{\tor}$ and the cokernel of $\alpha_M$ is pseudo-null, and that those properties characterizes $\alpha_M$ up to isomorphism.
For that reason, by slight abuse of notation, we say that a homomorphism $f$ from $M$ to a reflexive $\RR$-module is the reflexive hull of $M$ if the kernel of $f$ is $M_{\tor}$ and the cokernel of $f$ is pseudo-null.

The following lemma gives a relation between Proposition \ref{prop:30} and the reflexive hull.

\begin{lem}\label{lem:40}
In the situation of Proposition \ref{prop:30},
the map $\Psi_C$ is the reflexive hull of $\bigwedge_{\RR}^l H^1(C)$ if and only if $H^1(C)_{\tor}$ is pseudo-null.
\end{lem}

\begin{proof}
By definition, $\Psi_C$ is the reflexive hull of $\bigwedge_{\RR}^l H^1(C)$ if and only if $\Coker(\Psi_C)$ is pseudo-null.
For each prime $\qu$ of $\Lambda$ with $\height(\qu) \leq 1$, we have
\begin{align}
\Coker(\Psi_C)_{\qu} = 0 
& \Leftrightarrow \Fitt_{\RR_{\qu}}(E^1(H^1(C))_{\qu}) = \RR_{\qu}\\
& \Leftrightarrow E^1(H^1(C))_{\qu} = 0\\
&\Leftrightarrow \text{$H^1(C)_{\qu}$ is projective over $\Lambda_{\qu}$}\\
&\Leftrightarrow \text{$H^1(C)_{\qu}$ is torsion-free over $\Lambda_{\qu}$},
\end{align}
where the first equivalence follows from Proposition \ref{prop:30}(2) and the final from $\Lambda_{\qu}$ being either a field or a discrete valuation ring.
Hence the lemma follows.
\end{proof}

We also mention a relation with the notion of exterior power biduals, which is developed by Burns, Sano, and Sakamoto, in the theory of Euler systems (e.g., \cite{BS19}, \cite{Sak18}).
In general, for a finitely generated $\RR$-module $M$, we define the $l$-th exterior power bidual by
\[
\bigcap_{\RR}^l M = \left( \bigwedge_{\RR}^l M^* \right)^*.
\]
Then we have a natural map
\[
\alpha_M^l: \bigwedge_{\RR}^l M \to \bigcap_{\RR}^l M
\]
defined by
\[
x_1 \wedge \dots \wedge x_l \mapsto [\phi_1 \wedge \dots \wedge \phi_l \mapsto \det(\phi_i(x_j))_{ij}]
\]
for $x_1, \dots, x_l \in M$ and $\phi_1, \dots, \phi_l \in M^*$.
For example, when $l = 1$, then $\alpha_M^1$ is identified with the natural map $\alpha_M: M \to M^{**}$.

Now we prove the following.

\begin{prop}\label{prop:43}
In the situation of Proposition \ref{prop:30}, if moreover $H^1(C)_{\tor}$ is pseudo-null, then we have a natural isomorphism
\[
\bigcap_{\RR}^l H^1(C) \simeq \Det^{-1}_{\RR}(C).
\]
In particular, $\bigcap_{\RR}^l H^1(C)$ is free of rank one over $\RR$.
\end{prop}

\begin{proof}
Since $H^i(C)$ is torsion for any $i \neq 1$ and the exterior power bidual is torsion-free, we have a natural injective map
\[
\bigcap_{\RR}^l H^1(C) 
\hookrightarrow Q(\RR) \otimes_{\RR} \bigcap_{\RR}^l H^1(C) 
\simeq \Det_{Q(\RR)}^{-1}(Q(\RR) \otimesL_{\RR} C)
\simeq Q(\RR) \otimes_{\RR} \Det_{\RR}^{-1}(C).
\]
We shall show that the image of this composite map coincides with $\Det_{\RR}^{-1}(C)$.
For each prime ideal $\qu$ with $\height(\qu) \leq 1$, the assumption and the Auslander-Buchsbaum formula imply that $H^1(C)_{\qu}$ is free.
Hence the claim is true after localization at $\qu$.
Since both $\bigcap_{\RR}^l H^1(C)$ and $\Det_{\RR}^{-1}(C)$ are reflexive, we obtain the claim.
\end{proof}

The $l = 1$ case of this proposition will be used to prove Theorem \ref{thm:101}.
The author does not think that we can essentially simplify the proof of the proposition even if we assume $l = 1$.

Note that if $H^1(C)_{\tor}$ is pseudo-null, Proposition \ref{prop:43} implies that we can use $\bigcap_{\RR}^l H^1(C)$ instead of $\Det^{-1}_{\RR}(C)$ in \eqref{eq:56}.
However, we will have to deal with the case where $H^1(C)_{\tor}$ is not pseudo-null, and then this alternative is not available.

\section{Facts on complexes}\label{sec:14}

In this section, we review facts on arithmetic complexes.
Essentially the statements have already been written in \cite{BCG+} and \cite{BCG+b}, but we need a slight modification, for example, from $S_p(E)$ to larger $\Sigma_f$.
For more details, see also Nekov\'{a}\v{r}'s book \cite{Nek06}.

We consider the following situation.
Let $E$ be a number field (in this section we do not assume that $E$ is a CM-field).
Let $K$ be an abelian extension of $E$ that is a $\Z_p^r$-extension ($r \geq 1$) of a number field.
Suppose that $K$ contains $\mu_{p^{\infty}} = \bigcup_{n} \mu_{p^n}$.
Put $\RR = \Z_p[[\Gal(K/E)]]$.

\subsection{Generalities}\label{subsec:59}

As in the introduction, we write $S_p(E)$ and $S_{\infty}(E)$ as the sets of $p$-adic primes and infinite places of $E$, respectively.
We define $S_{\ram, p}(K/E)$ as the set of finite primes $v \not \in S_p(E)$ of $E$ such that the ramification index of $v$ in $K/E$ is divisible by $p$.
Let $\Sigma$ be a finite set of places of $E$ such that 
\begin{equation}\label{eq:60}
\Sigma \supset S_{\infty}(E) \cup S_p(E) \cup S_{\ram, p}(K/E).
\end{equation}
For example, if $p \nmid [K: \tilde{E}]$, then $S_{\ram, p}(K/E) = \emptyset$, so we only have to assume $\Sigma \supset S_{\infty}(E) \cup S_p(E)$.
This condition on $\Sigma$ will be needed in Proposition \ref{prop:a23} below.
For such a set $\Sigma$, we put $\Sigma_f = \Sigma \setminus S_{\infty}(E)$.

We define $K_{\Sigma}$ as the maximal algebraic extension of $K$ which is unramified at any prime not lying above a prime in $\Sigma$.

Let $T$ be a finitely generated free $\Z_p$-module equipped with an action of $\Gal(K_{\Sigma}/E)$.
(We will need only the case $T = \Z_p(1)$ in this paper.)
Let $\RG_{\Iw}(K_{\Sigma}/K, T)$ be the complex over $\RR$ that computes the global Iwasawa cohomology groups
\[
H^i_{\Iw}(K_{\Sigma}/K, T) = \varprojlim_{F} H^i(K_{\Sigma}/F, T),
\]
where $F$ runs over all intermediate number fields in $K/E$.
As local counterparts, for each finite prime $v$ of $E$, let $\RG_{\Iw}(K_{v}, T)$ be the complex over $\RR$ that computes the local Iwasawa cohomology groups
\[
H^i_{\Iw}(K_{v}, T) = \varprojlim_{F} H^i(F \otimes_E E_{v}, T).
\]
The actual construction of these complexes will be given in the proof of the next proposition.

\begin{prop}\label{prop:a23}
We have
\[
\RG_{\Iw}(K_{\Sigma}/K, T) \in \De^{[0, 2]}(\RR)
\]
and
\[
\RG_{\Iw}(K_{v}, T) \in \De^{[0, 2]}(\RR)
\]
for each finite prime $v$ of $E$.
\end{prop}

\begin{proof}
We can define the complexes concerned by
\[
\RG_{\Iw}(K_{\Sigma}/K, T) = \RG(K_{\Sigma}/E, T \otimes_{\Z_p} \RR)
\]
and
\[
\RG_{\Iw}(K_v, T) = \RG(E_v, T \otimes_{\Z_p} \RR),
\]
where $T \otimes_{\Z_p} \RR$ is regarded as a Galois representation of $\Gal(\ol{E}/E)$ over $\RR$ 
by defining the action of the Galois group on $\RR$ via
\[
\Gal(\ol{E}/E) \twoheadrightarrow \Gal(K/E) \hookrightarrow \RR^{\times} \overset{\iota}{\to} \RR^{\times},
\]
where the first two maps are the natural ones and $\iota$ denotes the involution which inverts every group element.
Then Shapiro's lemma implies the above-mentioned descriptions of the cohomology groups as inverse limits.

By \cite[Proposition 4.2.9]{Nek06}, it is therefore enough to show that $\cd_p \Gal(K_{\Sigma}/E) \leq 2$ and $\cd_p \Gal(\ol{E_v}/E_v) \leq 2$, where $\cd_p$ denotes the $p$-cohomological dimension.
The latter claim $\cd_p \Gal(\ol{E_v}/E_v) \leq 2$ is a well-known fact (see \cite[Theorem (7.1.8)]{NSW08}).
For $\cd_p \Gal(K_{\Sigma}/E) \leq 2$, let $E_1$ be the maximal intermediate number field of $K/E$ such that $p \nmid [E_1: E]$.
Thanks to assumption \eqref{eq:60}, the extension $K/E_1$ is unramified outside $\Sigma$.
Therefore, $K_{\Sigma}$ coincides with the maximal algebraic extension of $E_1$ which is unramified outside $\Sigma$.
Then we have $\cd_p \Gal(K_{\Sigma}/E_1) \leq 2$ (see \cite[Proposition (8.1.18)]{NSW08}).
Since $p \nmid [E_1:E]$, we obtain $\cd_p \Gal(K_{\Sigma}/E) \leq 2$ (see \cite[Proposition (3.3.5)]{NSW08}), as claimed.
\end{proof}

\begin{defn}\label{defn:32}
For each subset $S \subset \Sigma_f$, we define a complex $C_{\Sigma, S}(T)$ by a triangle
\begin{equation}\label{eq:16}
C_{\Sigma, S}(T) \to
\RG_{\Iw}(K_{\Sigma}/K, T) 
\to \bigoplus_{v \in S} \RG_{\Iw}(K_{v}, T),
\end{equation}
where the second map is the natural restriction morphisms.
\end{defn}

For example, we have
\[
C_{\Sigma, \emptyset}(T) \simeq \RG_{\Iw}(K_{\Sigma}/K, T).
\]
By the Poitou-Tate duality, we also have a triangle
\begin{equation}\label{eq:17}
C_{\Sigma, S}(T) 
\to \bigoplus_{v \in \Sigma_f \setminus S} \RG_{\Iw}(K_{v}, T)
\to \RG(K_{\Sigma}/K, T^{\dual}(1))^{\dual}[-2],
\end{equation}
so in particular
\[
C_{\Sigma, \Sigma_f}(T) \simeq \RG(K_{\Sigma}/K, T^{\dual}(1))^{\dual}[-3].
\]

The following duality theorem plays an important role.

\begin{prop}\label{prop:44}
Let $T^{\sharp}$ be the $\Z_p$-linear dual of $T$.
\begin{itemize}
\item[(1)]
For each subset $S \subset \Sigma_f$, we have a quasi-isomorphism
\[
\RHom_{\RR}(C_{\Sigma, S}(T), \RR)^{\iota} \simeq C_{\Sigma, \Sigma_f \setminus S}(T^{\sharp}(1))[3].
\]
\item[(2)]
For each finite prime $v$ of $E$, we have a quasi-isomorphism
\[
\RHom_{\RR}(\RG_{\Iw}(K_v, T), \RR)^{\iota} \simeq \RG_{\Iw}(K_v, T^{\sharp}(1))[2]
\]
\end{itemize}
\end{prop}

\begin{proof}
See \cite[Proposition 2.1]{BCG+b}.
\end{proof}

\begin{prop}\label{prop:18}
If $S = \emptyset$, we have $C_{\Sigma, \emptyset}(T) \in \De^{[0, 2]}(\RR)$.
If $S = \Sigma_f$, we have $C_{\Sigma, \Sigma_f}(T) \in \De^{[1,3]}(\RR)$.
If $\emptyset \subsetneqq S \subsetneqq \Sigma_f$, we have $C_{\Sigma, S}(T) \in \De^{[1, 2]}(\RR)$.
\end{prop}

\begin{proof}
By Proposition \ref{prop:a23}, we have $C_{\Sigma, S}(T) \in \De^{[0, 3]}(\RR)$ in general.
If $S \subsetneqq \Sigma_f$, then \eqref{eq:17} implies that $H^3(C_{\Sigma, S}(T)) = 0$, so we have $C_{\Sigma, S}(T) \in \De^{[0, 2]}(\RR)$.
If $S \neq \emptyset$, then Proposition \ref{prop:44}(1) and $C_{\Sigma, \Sigma_f \setminus S}(T^{\sharp}(1)) \in \De^{[0, 2]}(\RR)$ imply that $C_{\Sigma, S}(T) \in \De^{[1, 3]}(\RR)$.
Then the proposition follows.
\end{proof}

\subsection{Applications to $T = \Z_p(1)$}\label{subsec:60}

We specialize the facts in \S \ref{subsec:59} to $T = \Z_p(1)$.

For each finite prime $v$ of $E$, we put $K_v = K \otimes_E E_v$.
Then we define $\RR$-modules $D_v(K)$ and $Z_v(K)$ by
\[
D_{v}(K) = H^1(K_{v}, \Q_p/\Z_p)^{\dual}
\]
and
\[
Z_{v}(K) = H^0(K_{v}, \Q_p/\Z_p)^{\dual} \simeq \Z_p[[\Gal(K/E)/G_{v}(K/E)]],
\]
where $G_{v}(K/E)$ denotes the decomposition group of $K/E$ at $v$.
For each finite set $S$ of finite primes of $E$, we also put
\[
D_{S}(K) = \bigoplus_{v \in S} D_{v}(K),
\qquad Z_{S}(K) = \bigoplus_{v \in S} Z_{v}(K).
\]

For a finite set $S$ of finite primes of $E$, let $\XX_{S}(K)$ denote the Iwasawa module defined as the Galois group over $K$ of the maximal abelian pro-$p$ extension of $K$ which is totally split outside $S$.
It is known that $\XX_S(K)$ is a finitely generated $\RR$-module.
Also, we put $\XX(K) = \XX_{\emptyset}(K)$.

Note that, concerning the Iwasawa modules, we are considering the ``totally split'' condition rather than the ``unramified'' condition.
Since $K$ contains the cyclotomic $\Z_p$-extension of $E$, the residue degree of $K/E$ at any non-$p$-adic prime is divisible by $p^{\infty}$.
Moreover, if $E$ is a $p$-ordinary CM-field and $K$ contains $\tilde{E}$, then the residue degree at any $p$-adic prime is also divisible by $p^{\infty}$ by \cite[Lemma 3.1(ii)]{BCG+b}.
Therefore, the notation $X_S(K)$ is compatible with that in the introduction.

Now we review descriptions of the local and global cohomology groups (see, e.g., \cite[\S 2]{BCG+b}).
We put
\[
C_v^{\loc} = \RG_{\Iw}(K_v, \Z_p(1))
\]
for each finite prime $v$ of $E$ and also put $C_S^{\loc} = \bigoplus_{v \in S} C_v^{\loc}$.
Then by the local duality, the local cohomology groups are described as
\begin{equation}\label{eq:42}
H^i(C_v^{\loc}) \simeq
\begin{cases}
	D_{v}(K) & (i=1)\\
	Z_{v}(K) & (i = 2)\\
	0 & (i \neq 1, 2).
\end{cases}
\end{equation}

For each subset $S \subset \Sigma_f$, we put
\[
C_{\Sigma, S} = C_{\Sigma, S}(\Z_p(1)).
\]
Then by the Poitou-Tate duality, we obtain the following description.
For $S = \Sigma_f$, we have
\begin{equation}\label{eq:75}
H^i(C_{\Sigma, \Sigma_f}) \simeq 
\begin{cases}
	\XX_{\Sigma_f}(K) & (i=2)\\
	\Z_p & (i = 3)\\
	0 & (i \neq 2, 3).
\end{cases}
\end{equation}
If $S \subsetneqq \Sigma_f$, then we have $H^i(C_{\Sigma, S}) = 0$ for $i \neq 1, 2$ and also exact sequences 
\begin{equation}\label{eq:76}
0 \to H^1(C_{\Sigma, S}) \to H^1_{\Iw}(K_{\Sigma}/K, \Z_p(1)) 
\to \bigoplus_{v \in S} H^1_{\Iw}(K_{v}, \Z_p(1)) 
\end{equation}
and
\begin{equation}\label{eq:77}
0 \to \XX_{S}(K) \to H^2(C_{\Sigma, S}) \to Z_{\Sigma_f \setminus S}^0(K) \to 0.
\end{equation}
Here, when $S$ is a nonempty finite set of finite primes of $E$, we put
\[
Z_S^0(K) = \Ker(Z_S(K) \to \Z_p),
\]
the augmentation kernel.
In particular, \eqref{eq:77} for $S = \emptyset$ implies
\[
0 \to \XX(K) \to H^2_{\Iw}(K_{\Sigma}/K, \Z_p)(1) \to Z_{\Sigma_f}^0(K) \to 0
\]
as mentioned just after Theorem \ref{thm:102}.

Finally we record the results of Proposition \ref{prop:44}:
because of $\Z_p(1)^{\sharp}(1) \simeq \Z_p \simeq (\Z_p(1))(-1)$,
we have isomorphisms
\begin{equation}\label{eq:93}
\RHom_{\RR}(C_{\Sigma, S}, \RR)^{\iota}(1)[-3] \simeq C_{\Sigma, \Sigma_f \setminus S}
\end{equation}
and
\begin{equation}\label{eq:94}
\RHom_{\RR}(C_v^{\loc}, \RR)^{\iota}(1)[-2] \simeq C_v^{\loc}.
\end{equation}

\section{Algebraic $p$-adic $L$-functions associated to CM-fields}\label{sec:23}

In the rest of this paper, we consider the same situation as \S \ref{subsec:24}:
$E$ is a CM-field that satisfies the $p$-ordinary condition and $K$ is an abelian extension of $E$ which is a finite extension of $\tilde{E}(\mu_p)$.
We put $2d = [E: \Q]$.
We also assume the following.

\begin{ass}\label{ass:41}
For each CM-type $\cS$ of $E$, the module $\XX_{\cS}(K)$ is torsion over $\RR$.
\end{ass}

\begin{rem}\label{rem:92}
Assumption \ref{ass:41} is a fundamental property to study Iwasawa theory for CM-fields.
In fact, it is claimed to be true in general by \cite[Theorem 1.2.2]{HT94}.
However, the proof does not seem complete because the auxiliary algebraic lemma \cite[Lemma 1.2.4]{HT94}, which is stated without proof, has counter-examples when the closed subgroup $H$ has torsion.

On the other hand, Assumption \ref{ass:41} holds when either $K$ is a $\Z_p^r$-extension of a CM-field or $E$ is an imaginary quadratic field, for the following reasons.
When $K$ is a $\Z_p^r$-extension of CM-field, the proof of \cite[Theorem 1.2.2]{HT94} is valid because we only have to apply \cite[Lemma 1.2.4]{HT94} for $H$ without torsion.
When $E$ is an imaginary quadratic field, the proof is easier since the $\pe$-adic Leopoldt conjecture is known to be true for finite abelian extensions of $E$.
\end{rem}

Let $\Sigma$ be a set satisfying \eqref{eq:60}.
The Euler characteristics of the arithmetic complexes defined in \S \ref{subsec:60} are computed as follows.
For each prime $\pe \in S_p(E)$, we put $\deg(\pe) = [E_{\pe}:\Q_p]$.

\begin{lem}\label{lem:42}
For a subset $S \subset \Sigma_f$, 
we have
\[
\chi_{\RR}(C_{\Sigma, S}) = d - \sum_{\pe \in S \cap S_p(E)} \deg(\pe).
\]
\end{lem}

\begin{proof}
By the global and local Euler characteristic formulas (see, e.g., \cite[(7.3.1), (8.7.4)]{NSW08}), we obtain the following.
\begin{itemize}
\item We have $\chi_{\RR}(\RG_{\Iw}(K_{\Sigma}/K, \Z_p(1))) = d$.
\item For each $\pe \in S_p(E)$, we have $\chi_{\RR}(\RG_{\Iw}(K_{\pe}, \Z_p(1))) = \deg(\pe)$.
\item For each finite prime $v$ of $E$ not lying above $p$, we have $\chi_{\RR}(\RG_{\Iw}(K_{v}, \Z_p(1))) = 0$.
\end{itemize}
Then the lemma follows from the definition of $C_{\Sigma, S}$.
\end{proof}

\begin{lem}\label{lem:37}
Let $S \subset \Sigma_f$ be a subset such that $S \cap S_p(E)$ is a CM-type for $E$.
Then we have $C_{\Sigma, S} \in \De^{[1, 2]}(\RR)$, 
$H^1(C_{\Sigma, S}) = 0$, and $H^2(C_{\Sigma, S})$ is torsion.
\end{lem}

\begin{proof}
By Proposition \ref{prop:18}, we have $C_{\Sigma, S} \in \De^{[1, 2]}(\RR)$.
Since $\XX_S(K)$ is torsion by Assumption \ref{ass:41}, the exact sequence \eqref{eq:77} implies that $H^2(C_{\Sigma, S})$ is torsion.
Then we also deduce $H^1(C_{\Sigma, S}) = 0$ from $C_{\Sigma, S} \in \De^{[1, 2]}(\RR)$ and $\chi_{\RR}(C_{\Sigma, S}) = 0$ by Lemma \ref{lem:42}.
This completes the proof.
\end{proof}

In general, for a complex $C \in \DeP(\RR)$ whose cohomology groups are all torsion, we have a trivialization map
\[
\iota_C: \Det_{\RR}^{-1}(C) \hookrightarrow \Det_{Q(\RR)}^{-1}(Q(\RR) \otimesL_{\RR} C) \simeq Q(\RR),
\]
where the isomorphism follows from the assumption that $Q(\RR) \otimesL_{\RR} C$ is acyclic.
We put 
\[
\iDet_{\RR}(C) = \iota_C(\Det_{\RR}^{-1}(C)),
\]
which is an invertible fractional ideal of $\RR$.
If moreover $C \in \De^{[1, 2]}(\RR)$ and $H^1(C) = 0$, then we have 
\[
\iDet_{\RR}(C) = \Fitt_{\RR}(H^2(C)) \subset \RR.
\]
This is because the $C$ can be regarded as a finite presentation of $H^2(C)$ (we omit the detail; see \cite[\S 3]{Kata_10}).
Therefore, we can define the algebraic $p$-adic $L$-function as follows.

\begin{defn}\label{defn:38}
Let $S \subset \Sigma_f$ be a subset such that $S \cap S_p(E)$ is a CM-type for $E$.
Then we define the algebraic $p$-adic $L$-function $\LL^{\alg}_{\Sigma, S} \in \RR$ (defined up to a unit) as a generator of $\iDet_{\RR}(C_{\Sigma, S}) = \Fitt_{\RR}(H^2(C_{\Sigma, S}))$.
\end{defn}

By Lemma \ref{lem:37} and \eqref{eq:77}, the element $\LL^{\alg}_{\Sigma, S} \in \RR$ has information on $\XX_{S}$ (up to the easy factor $Z_{\Sigma_f \setminus S}^0(K)$).
It is then natural to formulate a main conjecture as a relation between $\LL^{\alg}_{\Sigma, S}$ and a certain analytically defined $p$-adic $L$-function.
In fact, in the article \cite{Kata_10} of the author, such a kind of main conjecture is obtained when $E$ is an imaginary quadratic field.
However, in this paper we do not study the analytic aspects and focus on the algebraic side only.

\section{The first main theorem}\label{sec:33}

\subsection{A choice of CM-types}\label{subsec:choice}

In our two main theorems, we consider the following situation (motivated by the work \cite{BCG+b}).
We keep the notations in \S \ref{sec:23}.

Let $\VV$ be a set satisfying $S_p(E) \subset \VV \subset \Sigma_f$.
Let $\cS_1, \dots, \cS_n$ be distinct CM-types for $E$ with $n \geq 2$.
Put 
\[
\UU = \bigcup_{i=1}^n \cS_i, 
\qquad \TT_i = \UU \setminus \cS_i
\]
and put 
\[
l = \sum_{\pe \in \TT_i} \deg(\pe) = \sum_{\pe \in \UU} \deg(\pe) - d.
\]

For a subset $S$ of $S_p(E)$, we put $S^c = S_p(E) \setminus S$.
Note that a CM-type $\cS$ satisfies $\cS^c = \ol{\cS}$.
The following is a motivation for introducing the integer $l$.

\begin{lem}\label{lem:rank}
The following are true.
\begin{itemize}
\item[(1)]
For each $1 \leq i \leq n$, 
the module $D_{\TT_i}(K)$  is of rank $l$ over $\RR$.
\item[(2)]
The module $\XX_{\VV \setminus \UU^c}(K)$ is of rank $l$ over $\RR$.
\end{itemize}
\end{lem}

\begin{proof}
(1)
For each $\pe \in \TT_i$, we have the complex $C_{\pe}^{\loc}$ satisfying \eqref{eq:42}.
Since $Z_{\pe}(K)$ is torsion (see also Lemma \ref{lem:a13} below) and $\chi_{\RR}(C_{\pe}^{\loc}) = \deg(\pe)$ as in the proof of Lemma \ref{lem:42}, the rank of $D_{\pe}(K)$ is $\deg(\pe)$.
Then the claim follows.

(2)
We have the complex $C_{\Sigma, \VV \setminus \UU^c}$ whose second cohomology has the same rank as $X_{\VV \setminus \UU^c}(K)$ by \eqref{eq:75} and \eqref{eq:77}.
By Lemma \ref{lem:42}, we have $\chi_{\RR}(C_{\Sigma, \VV \setminus \UU^c}) = d - \sum_{\pe \in \UU} \deg(\pe) = - l$.
We also have $H^1(C_{\Sigma, \VV \setminus \UU^c}) = 0$ since, for any $i$, we have $H^1(C_{\Sigma, \VV \setminus \UU^c}) \subset H^1(C_{\Sigma, \VV \setminus \cS_i^c}) = 0$ by Lemma \ref{lem:37}.
Then the claim follows. 
\end{proof}

We also record a lemma that will be used later.

\begin{lem}\label{lem:a13}
For each $\pe \in S_p(E)$, the module $Z_{\pe}(K)$ is pseudo-null over $\RR$.
\end{lem}

\begin{proof}
By \cite[Lemma 3.1(i)]{BCG+b}, the $\Z_p$-rank $r_{\fp}$ of $G_{\pe}(\tilde{E}/E)$ satisfies $r_{\fp} = 1 + \deg(\fp)$.
In particular, we have $r_{\fp} \geq 2$, which implies the lemma.
\end{proof}

\subsection{The statement and the proof}\label{subsec:131}

Let us take an auxiliary intermediate number field $E'$ of $K/E$ such that 
$E'/E$ is a $p$-extension and $\Gal(K/E')$ is $p$-torsion-free.
Put $\Lambda = \Z_p[[\Gal(K/E')]]$, which is a subring of $\RR$.
For example, if $p \nmid [K: \tilde{E}]$, then we may take $E' = E$ and then $\Lambda = \RR$.
Note that $\Lambda$ is a finite product of regular local rings.
For a prime ideal $\qu$ of $\Lambda$, the localization $\Lambda_{\qu}$ is then a regular local ring, but $\RR_{\qu}$ is not even a domain in general.

The following is the first main theorem of this paper. 

\begin{thm}\label{thm:103}
Let $\qu$ be a prime ideal of $\Lambda$ such that 
\[
\pd_{\RR_{\qu}}(Z_{\pe}(K)_{\qu}) \leq 2, 
\qquad Z_{\pe}(K)^{\iota}(1)_{\qu} = 0
\]
for every $\pe \in \TT_i$ ($1 \leq i \leq n$).
Then we have an exact sequence
\[
0 \to \left(\frac{\left(\bigwedge_{\RR}^l H^2(C_{\Sigma, \VV \setminus \UU^c}) \right)_{/\tor}}
{\sum_{i=1}^n \bigwedge_{\RR}^l D_{\TT_i}(K)}\right)_{\qu}
\to \frac{\RR_{\qu}}{\sum_{i=1}^n (\LL^{\alg}_{\Sigma, \VV \setminus \cS_i^c})}
\to \frac{\RR_{\qu}}{\Fitt_{\RR_{\qu}}(H^2(C_{\Sigma, (\Sigma_f \setminus \VV) \cup \UU^c})^{\iota}(1)_{\qu})}
\to 0.
\]
\end{thm}

Recall that the descriptions of $H^2(C_{\Sigma, \VV \setminus \UU^c})$ and of $H^2(C_{\Sigma, (\Sigma_f \setminus \VV) \cup \UU^c})$ are given in \S \ref{subsec:60}.
We immediately deduce Theorem \ref{thm:102} from Theorem \ref{thm:103} by putting $\VV = \Sigma_f$, $n = 2$, $\cS_1 = \cS$, and $\cS_2 = \ol{\cS}$. 

Concerning the condition $\pd_{\RR_{\qu}}(Z_{\pe}(K)_{\qu}) \leq 2$ in Theorem \ref{thm:103}, we observe the following.

\begin{prop}\label{prop:80'}
Let $\pe \in S_p(E)$ and let $\qu$ be a prime ideal of $\Lambda$.
We have $\pd_{\RR_{\qu}}(Z_{\pe}(K)_{\qu}) = - \infty$ unless $\qu \supset \Ann_{\Lambda}(Z_{\pe}(K))$.
If $\qu \supset \Ann_{\Lambda}(Z_{\pe}(K))$ holds, then we have
\[
\pd_{\RR_{\qu}}(Z_{\pe}(K)_{\qu})
= \begin{cases}
	\deg(\pe) + 1 & (\text{if $p \not \in \qu$ or $G_{\pe}(K/E)$ is $p$-torsion-free})\\
	+ \infty & (\text{if $p \in \qu$ and $G_{\pe}(K/E)$ is not $p$-torsion-free}).
\end{cases}
\]
Moreover, $G_{\pe}(K/E)$ is $p$-torsion-free unless $E_{\pe}$ contains a primitive $p$-th root of unity.
\end{prop}

\begin{proof}
We write $\cG = \Gal(K/E)$, $\cG' = \Gal(K/E')$, $\cG_0 = G_{\fp}(K/E)$, and $\cG_0' = \cG' \cap \cG_0$.
Then we have $\RR = \Z_p[[\cG]]$, $\Lambda = \Z_p[[\cG']]$, and $Z_{\fp}(K) = \Z_p[[\cG/\cG_0]]$.
We also put $\RR_0 = \Z_p[[\cG_0]]$ and $\Lambda_0 = \Z_p[[\cG_0']]$.
As we used in Lemma \ref{lem:a13}, by \cite[Lemma 3.1(i)]{BCG+b}, the $\Z_p$-rank $r_{\fp}$ of the $p$-component of $\cG_0$ is $r_{\fp} = \deg(\fp) + 1$.

If $\qu \not \supset \Ann_{\Lambda}(\Z_p[[\cG/\cG_0]])$, then $\Z_p[[\cG/\cG_0]]_{\qu} = 0$, so the claim is clear.
Let us assume $\qu \supset \Ann_{\Lambda}(\Z_p[[\cG/\cG_0]])$.
First we show the inequality $\leq$ of the displayed assertion.
Note that $\Z_p[[\cG/\cG_0]] \simeq \Z_p \otimes_{\RR_0} \RR$ implies $\Z_p[[\cG/\cG_0]]_{\qu} \simeq (\Z_p)_{\qu_0} \otimes_{(\RR_0)_{\qu_0}} \RR_{\qu}$, where we put $\qu_0 = \qu \cap \Lambda_0$.
Therefore, we have
\[
\pd_{\RR_{\qu}}(\Z_p[[\cG/\cG_0]]_{\qu}) \leq \pd_{(\RR_0)_{\qu_0}}((\Z_p)_{\qu_0}).
\]
Using the well-known description of $\RR_0$ as a ring of power series, 
we see that $\pd_{\RR_0}(\Z_p) = r_{\fp}$ if $\cG_0$ is $p$-torsion-free.
We also have $\pd_{\RR_0[1/p]}(\Q_p) = r_{\fp}$, which implies $\pd_{(\RR_0)_{\qu_0}}((\Z_p)_{\qu_0}) \leq r_{\fp}$ if $p \not \in \qu$.
Therefore, we have the inequality $\leq$ of the displayed assertion.

Next we show the opposite inequality $\geq$.
For each $i \geq 0$, we have
\[
\Ext^i_{\RR}(\Z_p[[\cG/\cG_0]], \Z_p[[\cG/\cG_0]])
\simeq \Ext^i_{\RR_0}(\Z_p, \Z_p) \otimes_{\RR_0} \RR,
\]
so
\[
\Ext^i_{\RR_{\qu}}(\Z_p[[\cG/\cG_0]]_{\qu}, \Z_p[[\cG/\cG_0]]_{\qu})
\simeq \Ext^i_{\RR_0}(\Z_p, \Z_p)_{\qu_0} \otimes_{(\RR_0)_{\qu_0}} \RR_{\qu}.
\]
If $p \not \in \qu$ or $\cG_0$ is $p$-torsion-free, then the description of $\RR_0$ as a ring of power series again implies $\Ext^{r_{\fp}}_{\RR_0}(\Z_p, \Z_p)_{\qu_0} \simeq (\Z_p)_{\qu_0}$.
This does not vanish because the assumption $\qu \supset \Ann_{\Lambda}(\Z_p[[\cG/\cG_0]])$ implies $\qu_0 \supset \Ann_{\Lambda_0}(\Z_p)$.
If $\cG_0$ is not $p$-torsion-free,
then we can show that $\Ext^i_{\RR_0}(\Z_p, \Z_p)$ is a non-zero finite module for arbitrarily large $i \geq 0$, which implies $\Ext^i_{\RR_0}(\Z_p, \Z_p)_{\qu_0} \neq 0$ if moreover $p \in \qu_0$.
Therefore, we obtain the claimed inequality.

We show the final assertion.
By class field theory, the decomposition group $G_{\fp}(K/E)$ is isomorphic to a quotient of the profinite completion $\widehat{E_{\pe}^{\times}}$ of $E_{\pe}^{\times}$.
If $E_{\fp}$ does not contain a primitive $p$-th root of unity, then the pro-$p$-component of $\widehat{E_{\pe}^{\times}}$ is a free $\Z_p$-module of rank $\deg(\pe) + 1$.
Combining this with the fact $r_{\fp} = \deg(\pe) + 1$, we conclude that the pro-$p$-component of $G_{\pe}(K/E)$ is a free $\Z_p$-module of rank $\deg(\pe) + 1$.
This completes the proof.
\end{proof}

\begin{cor}\label{cor:a55}
We have $\pd_{\RR_{\qu}}(Z_{\fp}(K)_{\qu}) \leq 2$ if and only if 
either $\qu \not \supset \Ann_{\Lambda}(Z_{\fp}(K))$ or $\deg(\fp) = 1$.
\end{cor}

\begin{proof}
If $\deg(\fp) = 1$, then $E_{\fp} \simeq \Q_p$ does not contain a primitive $p$-th root of unity (since $p \geq 3$).
Therefore, this corollary follows from Proposition \ref{prop:80'}.
\end{proof}

The proof of Theorem \ref{thm:103} occupies the rest of this subsection.
The following is the key diagram (see \eqref{eq:56}).

\begin{prop}\label{prop:46}
We have a commutative diagram with exact rows
\[
\xymatrix{
	0 \ar[r]
	& \bigoplus_{i=1}^n \left(\bigwedge_{\RR}^l H^1(C^{\loc}_{\TT_i})\right)_{/\tor} \ar[r]^-{\oplus \Psi} \ar[d]_{f_1}
	& \bigoplus_{i=1}^n  \Det^{-1}_{\RR}(C^{\loc}_{\TT_i})  \ar[d]_{f_2} \ar[r]
	& \bigoplus_{i=1}^n  \Coker (\Psi_{C^{\loc}_{\TT_i}}) \ar[r] \ar[d]_{f_3}
	& 0\\
	0 \ar[r]
	& \left(\bigwedge_{\RR}^l H^2(C_{\Sigma, \VV \setminus \UU^c}) \right)_{/\tor} \ar[r]_-{\Psi}
	& \Det_{\RR}(C_{\Sigma, \VV \setminus \UU^c}) \ar[r]
	& \Coker (\Psi_{C_{\Sigma, \VV \setminus \UU^c}[1]}) \ar[r]
	& 0.
}
\]
\end{prop}

\begin{proof}
First, in order to construct the lower sequence, we show that Proposition \ref{prop:30} is applicable to $C = C_{\Sigma, \VV \setminus \UU^c}[1]$.
If $\UU^c = \emptyset$ and $\VV = \Sigma_f$, then we know by \eqref{eq:75} that $H^i(C_{\Sigma, \Sigma_f})$ is pseudo-null unless $i = 2$ (in fact it vanishes except for $H^3 \simeq \Z_p$).
If either $\UU^c \neq \emptyset$ or $\VV \subsetneqq \Sigma_f$, then by \eqref{eq:76} and Lemma \ref{lem:37}, we have $H^i(C_{\Sigma, \VV \setminus \UU^c}) = 0$ unless $i = 2$.
Thus the assumptions of Proposition \ref{prop:30} hold for $C = C_{\Sigma, \VV \setminus \UU^c}[1]$.
Moreover, $l = \chi_{\RR}(C_{\Sigma, \VV \setminus \UU^c}[1])$ holds by Lemma \ref{lem:42}.
Therefore, we have the homomorphism
\[
\Psi_{C_{\Sigma, \VV \setminus \UU^c}[1]}:
\bigwedge_{\RR}^l H^2(C_{\Sigma, \VV \setminus \UU^c}) 
\to \Det_{\RR}(C_{\Sigma, \VV \setminus \UU^c}).
\]
Note that Proposition \ref{prop:31} is also applicable unless $\UU^c = \emptyset$ and $\VV = \Sigma_f$.

Next we construct the upper sequence by applying Proposition \ref{prop:30} to $C = C^{\loc}_{\TT_i}$.
We know $H^i(C^{\loc}_{\TT_i})$ is pseudo-null unless $i = 1$ by the description \eqref{eq:42} and Lemma \ref{lem:a13}.
We also have $l = \chi_{\RR}(C^{\loc}_{\TT_i})$ by the formula in the proof of Lemma \ref{lem:42}.
Therefore, we have the homomorphism
\[
\Psi_{C^{\loc}_{\TT_i}}:
\bigwedge_{\RR}^l H^1(C^{\loc}_{\TT_i}) \to \Det^{-1}_{\RR}(C^{\loc}_{\TT_i}).
\]

Finally we construct the vertical arrows.
Since $\UU^c \cup \TT_i = \cS_i^c$, we have a triangle
\begin{equation}\label{eq:137}
C_{\Sigma, \VV \setminus \UU^c} \to C_{\Sigma, \VV \setminus \cS_i^c} \to C^{\loc}_{\TT_i}
\end{equation}
by Definition \ref{defn:32}.
Then the arrow $f_1$ is induced by the induced connecting homomorphism between $H^1$ and $H^2$.
The arrow $f_2$ is defined by
\[
\Det^{-1}_{\RR}(C^{\loc}_{\TT_i}) 
\simeq \Det_{\RR}(C_{\Sigma, \VV \setminus \UU^c}) 
\otimes_{\RR} \Det^{-1}_{\RR}(C_{\Sigma, \VV \setminus \cS_i^c})
\hookrightarrow \Det_{\RR}(C_{\Sigma, \VV \setminus \UU^c}),
\]
where the last arrow is induced by the trivialization map
\[
\iota_{C_{\Sigma, \VV \setminus \cS_i^c}}: \Det^{-1}_{\RR}(C_{\Sigma, \VV \setminus \cS_i^c}) \hookrightarrow \RR
\]
introduced just before Definition \ref{defn:38}.
The commutativity of the diagram is clear and we define $f_3$ as the induced one.
\end{proof}

We study the diagram in Proposition \ref{prop:46}.

\begin{prop}\label{prop:47}
We have an isomorphism
\[
\Coker(f_2) \simeq \frac{\RR}{\sum_{i=1}^n (\LL^{\alg}_{\Sigma, \VV \setminus \cS_i^c})}.
\]
\end{prop}

\begin{proof}
By the construction of $f_2$ we immediately obtain
\[
\Coker(\Det^{-1}_{\RR}(C^{\loc}_{\TT_i}) \hookrightarrow \Det_{\RR}(C_{\Sigma, \VV \setminus \UU^c}))
\simeq \RR / (\LL^{\alg}_{\Sigma, \VV \setminus \cS_i^c}).
\]
Then the proposition follows.
\end{proof}

\begin{prop}\label{prop:45}
Let $\qu$ be a prime ideal of $\Lambda$.
\begin{itemize}
\item[(1)]
In case $\UU^c = \emptyset$ and $\VV = \Sigma_f$, we suppose that $(\Z_p)_{\qu} = 0$ (i.e., $\qu$ does not contain the augmentation ideal of $\Lambda$).
Then we have
\[
\Coker (\Psi_{C_{\Sigma, \VV \setminus \UU^c}[1]})_{\qu}
\simeq \RR_{\qu} / \Fitt_{\RR_{\qu}}(H^2(C_{\Sigma, (\Sigma_f \setminus \VV) \cup \UU^c})^{\iota}(1)_{\qu}).
\]
\item[(2)]
Suppose that $Z_{\pe}(K)_{\qu} = 0$ for each $\pe \in \TT_i$.
Then we have
\[
\Coker (\Psi_{C^{\loc}_{\TT_i}})_{\qu} \simeq \RR_{\qu} / \Fitt_{\RR_{\qu}}(Z_{\TT_i}(K)^{\iota}(1)_{\qu}).
\]
\item[(3)]
Suppose that $\pd_{\RR_{\qu}}(Z_{\pe}(K)_{\qu}) \leq 2$ and that $Z_{\pe}(K)^{\iota}(1)_{\qu} = 0$ for each $\pe \in \TT_i$.
Then $(\Psi_{C^{\loc}_{\TT_i}})_{\qu}$ is an isomorphism.
\end{itemize}
\end{prop}

\begin{proof}
(1)
By the assumption, we have $H^1(C_{\Sigma, \VV \setminus \UU^c}) = 0$ and $H^3(C_{\Sigma, \VV \setminus \UU^c})_{\qu} = 0$ (when $\UU^c \neq \emptyset$ or $\VV \subsetneqq \Sigma_f$, we do not have to take the localization).
Hence we can apply Proposition \ref{prop:30}(2) to obtain
\[
\Coker (\Psi_{C_{\Sigma, \VV \setminus \UU^c}[1]})_{\qu} 
\simeq \RR_{\qu} / \Fitt_{\RR_{\qu}}(E^1(H^2(C_{\Sigma, \VV \setminus \UU^c}))_{\qu}).
\]
By \eqref{eq:93}, we also have
\[
E^1(H^2(C_{\Sigma, \VV \setminus \UU^c}))_{\qu} 
\simeq H^2(C_{\Sigma, (\Sigma_f \setminus \VV) \cup \UU^c})^{\iota}(1)_{\qu}.
\]

(2)
Since $H^0(C^{\loc}_{\TT_i}) = 0$ and $H^2(C^{\loc}_{\TT_i})_{\qu} = 0$ by assumption, we can apply Proposition \ref{prop:30}(2) to obtain
\[
\Coker (\Psi_{C^{\loc}_{\TT_i}})_{\qu} 
\simeq \RR_{\qu} / \Fitt_{\RR_{\qu}}(E^1(H^1(C^{\loc}_{\TT_i}))_{\qu}).
\]
By \eqref{eq:94}, we also have
\[
E^1(H^1(C^{\loc}_{\TT_i}))_{\qu} 
\simeq H^2(C_{\TT_i}^{\loc})^{\iota}(1)_{\qu}
\simeq Z_{\TT_i}(K)^{\iota}(1)_{\qu}.
\]

(3) Since $H^2((C_{\TT_i}^{\loc})^{\iota}(1))_{\qu} = 0$ by assumption, we have $(C_{\TT_i}^{\loc})^{\iota}(1)_{\qu} \in \De^{[0, 1]}(\RR_{\qu})$.
By the duality \eqref{eq:94}, it follows that $(C_{\TT_i}^{\loc})_{\qu} \in \De^{[1, 2]}(\RR_{\qu})$.
We also have $\pd_{\RR_{\qu}}(H^2(C_{\TT_i}^{\loc})_{\qu}) \leq 2$ by assumption.
Therefore, we may apply Proposition \ref{prop:30}(3).
\end{proof}

\begin{proof}[Proof of Theorem \ref{thm:103}]
The snake lemma applied to the diagram in Proposition \ref{prop:46}, together with Proposition \ref{prop:47}, gives us an exact sequence
\begin{align}\label{eq:5.9}
0 \to \Coker \left(\Ker(f_2) \to \Ker(f_3) \right)
& \to \frac{\left(\bigwedge_{\RR}^l H^2(C_{\Sigma, \VV \setminus \UU^c}) \right)_{/\tor}}
{\sum_{i=1}^n \bigwedge_{\RR}^l D_{\TT_i}(K)}\\
& \qquad \to \frac{\RR}{\sum_{i=1}^n (\LL^{\alg}_{\Sigma, \VV \setminus \cS_i^c})}
\to \Coker(f_3)
\to 0.
\end{align}
By Proposition \ref{prop:45}(3), the domain of $f_3$ vanishes after localization at $\qu$.
Moreover, we know the description of the target module of $f_3$ after localization at $\qu$ by Proposition \ref{prop:45}(1).
Therefore, Theorem \ref{thm:103} follows from \eqref{eq:5.9} and localization at $\qu$.
\end{proof}

Note that sequence \eqref{eq:5.9} can be viewed as a more general version of Theorem \ref{thm:103} that incorporates the idea of \cite[Theorem 5.9]{BCG+b}.
However, the descriptions of $\Ker(f_2)$, $\Ker(f_3)$, and $\Coker(f_3)$ seem to be complicated, so we stated only Theorem \ref{thm:103}.

\subsection{How to recover the previous result}\label{subsec:132}

In order to illustrate a relation with the work \cite{BCG+b}, we recall the notion of higher Chern classes, which is a key idea in \cite{BCG+}, \cite{BCG+b}.
For the sake of brevity, we only define Chern classes over (finite products of) regular local rings.

\begin{defn}\label{defn:136}
Let $A$ be a finite product of regular local rings and $m$ a positive integer.
We define $Z^m(A)$ as the free $\Z$-module on the set of height $m$ prime ideals of $A$.
For each $A$-module $M$ whose codimension is at least $m$, we define the $m$-th Chern class of $M$ by
\[
c_m(M) = \sum_{\qu} \length_{A_{\qu}}(M_{\qu}) [\qu] \in Z^m(A),
\]
where $\qu$ runs over the height $m$ primes of $A$.
\end{defn}

We now recall a main theorem of \cite{BCG+b} and will deduce it from Theorem \ref{thm:103}.
Recall that when $p \nmid [K: \tilde{E}]$, the ring $\RR$ is a product of regular local rings, so we have the notions of greatest common divisors and of characteristic ideals.

\begin{thm}[{\cite[Theorem 5.6]{BCG+b}}]\label{thm:81}
Suppose $p \nmid [K: \tilde{E}]$ and set $\Sigma_f = \VV = S_p(E)$.
Let $\theta \in \RR$ be a greatest common divisor of $\LL^{\alg}_{\Sigma, \cS_i}$ for $1 \leq i \leq n$.
Let $\theta_0 \in \RR$ be a generator of the characteristic ideal $\cha_{\RR}(\XX_{\UU}(K)_{\tor})$.
Then $\theta_0$ divides $\theta$ and we have 
\begin{align}
& c_2 \left(\frac{\RR}{\sum_{i=1}^n \theta^{-1} (\LL^{\alg}_{\Sigma, \cS_i})}\right)
\equiv  c_2
\left(\left(\frac{\left(\bigwedge_{\RR}^l \XX_{\UU}(K) \right)_{/\tor}}{\sum_{i=1}^n \bigwedge_{\RR}^l D_{\TT_i}(K)}\right)_{\PN}\right)
+ c_2 \left( \frac{\theta}{\theta_0} \frac{\RR}{\Fitt_{\RR}(E^2(\XX_{\UU^c}(K))^{\iota}(1))}\right)
\end{align}
in $Z^2(\RR)$, where the congruence means the equality outside the support of $Z_{\pe}(K)$ for $\pe \in \UU^c$ and that of $Z_{\pe}(K)^{\iota}(1)$ for $\pe \in \UU$.
\end{thm}

Before the proof, we observe the following proposition.

\begin{prop}\label{prop:20}
\begin{itemize}
\item[(1)]
We have an exact sequence
\begin{align}
0 & \to E^1(H^2(C_{\Sigma, (\Sigma_f \setminus \VV) \cup  \UU^c}))^{\iota}(1) 
\to H^2(C_{\Sigma, \VV \setminus \UU^c}) 
\to H^2(C_{\Sigma, \VV \setminus \UU^c})^{**}\\
& \to E^2(H^2(C_{\Sigma, (\Sigma_f \setminus \VV) \cup \UU^c}))^{\iota}(1)
\to W_{(\Sigma_f \setminus \VV) \cup \UU^c},
\end{align}
where in general we put $W_S = \Z_p$ if $S = \emptyset$ and $W_S = 0$ otherwise.
\item[(2)]
For each $\pe \in S_p(E)$, we have an exact sequence
\[
0 \to E^1(H^2_{\Iw}(K_{\pe}, \Z_p))^{\iota} 
\to D_{\pe}(K) \to D_{\pe}(K)^{**}
\to E^2(H^2_{\Iw}(K_{\pe}, \Z_p))^{\iota} 
\to 0.
\]
\end{itemize}
\end{prop}

\begin{proof}
This proposition is a direct generalization of \cite[Propositions 2.7 and 2.11]{BCG+b} (or \cite[Corollary 4.1.6 and Theorem 4.1.14]{BCG+} in a special case).
The key ingredient is the duality theorems \eqref{eq:93}, \eqref{eq:94}.
\end{proof}

\begin{proof}[Proof of Theorem \ref{thm:81}]
We deduce Theorem \ref{thm:81} from Theorem \ref{thm:103}.
Let $\qu$ be a prime ideal of $\RR$ with $\height(\qu) = 2$ outside the supports of the stated modules.
By Corollary \ref{cor:a55} and $r_{\fp} = \deg(\fp) + 1$, the assumption $\height(\qu) = 2$ implies $\pd_{\RR_{\qu}}(Z_{\fp}(K)_{\qu}) \leq 2$ for any $\fp \in S_p(E)$.
Therefore, the condition on $\qu$ required in Theorem \ref{thm:103} holds, so we may apply Theorem \ref{thm:103}.

By the assumption on $\qu$, we have isomorphisms
\[
\XX_{\UU^c}(K)^{\iota}(1)_{\qu} \simeq H^2(C_{\Sigma, \UU^c})^{\iota}(1)_{\qu}
\]
by \eqref{eq:77}, and
\[
\XX_{\UU}(K)_{\qu} \simeq H^2(C_{\Sigma, \UU})_{\qu}
\]
by \eqref{eq:75} and \eqref{eq:77} (note that $\Sigma_f \setminus \UU^c = \UU$).
Combining these with Proposition \ref{prop:20}(1), we obtain
\[
E^1(\XX_{\UU^c}(K))^{\iota}(1)_{\qu} \simeq (\XX_{\UU}(K)_{\tor})_{\qu}.
\]
Therefore,
\[
\cha_{\RR_{\qu}}(\XX_{\UU^c}(K)^{\iota}(1)_{\qu})
= \cha_{\RR_{\qu}}(E^1(\XX_{\UU^c}(K))^{\iota}(1)_{\qu})
= \theta_0 \RR_{\qu}.
\]
Then Proposition \ref{prop:91} implies that
\begin{align}\label{eq:133}
\Fitt_{\RR_{\qu}}(H^2(C_{\Sigma, \UU^c})^{\iota}(1)_{\qu}) 
&= \cha_{\RR_{\qu}}(\XX_{\UU^c}(K)^{\iota}(1)_{\qu}) 
\Fitt_{\RR_{\qu}}(E^2(\XX_{\UU^c}(K))^{\iota}(1)_{\qu})\\
&= \theta_0 \Fitt_{\RR_{\qu}}(E^2(\XX_{\UU^c}(K))^{\iota}(1)_{\qu}).
\end{align}

It is easy to see that
\[
\left( \frac{\RR}{\sum_{i=1}^n (\LL^{\alg}_{\Sigma, \cS_i})} \right)_{\PN} 
= \frac{\theta \RR}{\sum_{i=1}^n (\LL^{\alg}_{\Sigma, \cS_i})}
\simeq \frac{\RR}{\sum_{i=1}^n (\theta^{-1}\LL^{\alg}_{\Sigma, \cS_i})}.
\]
Hence Theorem \ref{thm:103} tells us an exact sequence
\[
0 \to \left(\frac{\left(\bigwedge_{\RR}^l \XX_{\UU}(K) \right)_{/\tor}}
{\sum_{i=1}^n \bigwedge_{\RR}^l D_{\TT_i}(K)}\right)_{\PN, \qu}
\to \frac{\RR_{\qu}}{\sum_{i=1}^n \theta^{-1} (\LL^{\alg}_{\Sigma, \cS_i})}
\to \theta \frac{\RR_{\qu}}{\Fitt_{\RR_{\qu}}(H^2(C_{\Sigma, \UU^c})^{\iota}(1)_{\qu})}
\to 0.
\]
By applying \eqref{eq:133} to the final module, we obtain the theorem.
\end{proof}

\begin{rem}\label{rem:K-gp}
As explained above, the main results of \cite{BCG+} and \cite{BCG+b} are formulated using the Chern classes of modules.
Let us now briefly discuss generalization of the notion of the Chern classes to equivariant settings.

For each positive integer $m$, let $\CC_{\RR}^m$ denote the category of finitely generated $\RR$-modules whose projective dimensions are finite and whose codimensions are at least $m$.
Note that, if $\RR$ is a finite product of regular local rings, then the finiteness of the projective dimension always holds.

Let us consider the Grothendieck group $K_0(\CC_{\RR}^m)$ of $\CC_{\RR}^m$ as an exact category.
Recall that $K_0(\CC_{\RR}^m)$ is defined by describing generators and relations: the generators are $[M]$ for objects $M$ and the relations are $[M] = [M'] + [M'']$ for exact sequences $0 \to M' \to M \to M'' \to 0$.
For $m = 1$, it is well-known that we have a natural isomorphism $K_0(\CC_{\RR}^1) \simeq Q(\RR)^{\times}/\RR^{\times}$, and this group is often used to formulate equivariant main conjectures in Iwasawa theory.
For general $m \geq 1$, if $p \nmid [K: \tilde{E}]$, one can also check that the map $[M] \mapsto c_m(M)$ gives an isomorphism $K_0(\CC_{\RR}^m) \simeq Z^m(\RR)$ (we omit the proof here).

These observations lead us to the idea of generalizing the Chern class $c_m(M) \in Z^m(\RR)$ in the non-equivariant settings to the class $[M] \in K_0(\CC_{\RR}^m)$ in the equivariant settings.
However, a crucial obstruction is that, in general, the modules in the main results of this paper (e.g., Theorem \ref{thm:102}) are not necessarily of finite projective dimension, so we cannot deduce a formula in the Grothendieck group.
Note also that, even in the non-equivariant settings, the exact sequences have more information than equations of the Chern classes.
Therefore, it is reasonable to formulate the main results in the form of exact sequences.

In \cite[\S 1.1]{BCG+}, the $m$-th Chern classes are also defined for complexes that are exact in codimension less than $m$.
This notion also seems to be amenable to equivariant situations, but it is unclear how to apply it to our arithmetic situation.
This is because we cannot expect that the arithmetic complexes concerned are exact in codimension less than $2$.
\end{rem}

\section{The second main theorem}\label{sec:11}

Our second main theorem focuses on the case where $l = 1$ in \S \ref{subsec:choice}.
Note that we have $l = 1$ if and only if $n = 2$ and $\TT_1$ (and thus $\TT_2$) consists of a unique prime of degree one.

Let $\Gal(K/\ttilde{E})^{(p')}$ denote the maximal subgroup of $\Gal(K/\ttilde{E})$ whose order is prime to $p$.
Let $\overline{\Q_p}$ be a fixed algebraic closure of $\Q_p$.
For each character $\psi: \Gal(K/\ttilde{E})^{(p')} \to \overline{\Q_p}^{\times}$, let $\RR^{\psi}$ denote the $\psi$-component of $\RR$.
Then we have a decomposition
\[
\RR = \prod_{\psi} \RR^{\psi},
\]
where $\psi$ runs over equivalence classes of characters of $\Gal(K/\ttilde{E})^{(p')}$; two characters $\psi, \psi'$ are said to be equivalent if $\psi' = \psi^{\sigma}$ for some $\sigma \in \Gal(\ol{\Q_p}/\Q_p)$.
Note that $\RR^{\psi}$ is a local ring but not a domain unless $\Gal(K/\ttilde{E})^{(p')} = \Gal(K/\ttilde{E})$, i.e., $p \nmid [K: \tilde{E}]$.

The following is the second main theorem, of which Theorem \ref{thm:101} is a special case.

\begin{thm}\label{thm:105}
Keep the notation in \S \ref{subsec:choice} and suppose $l = 1$.
Let $\psi$ be a character of $\Gal(K/\ttilde{E})^{(p')}$ such that 
\begin{equation}\label{eq:134}
Z_v(K)^{\omega\psi^{-1}} = 0
\end{equation}
 for any $v \in \VV \setminus S_p(E)$, and moreover that $\XX_{(\Sigma_f \setminus \VV) \cup \UU^c}(K)^{\omega \psi^{-1}}$ is pseudo-null.

\begin{itemize}
\item[(1)]
Suppose $d = 1$ and $\VV = S_p(E)$ (note that condition \eqref{eq:134} is trivial in this case).
We write $S_p(E) = \{\pe, \ol{\pe}\}$.
Then there exists an $\RR^{\psi}$-module $\cA$ which fits in an exact sequence
\[
0 \to \XX(K)^{\psi} \to \cA \to Z_{\Sigma_f \setminus S_p(E)}^0(K)^{\psi} \to 0
\]
such that we have an exact sequence
\[
\Z_p(1)^{\psi} 
\to \cA
\to \frac{\RR^{\psi}}{(\LL^{\alg, \psi}_{\Sigma, \{\pe\}}, \LL^{\alg, \psi}_{\Sigma, \{\ol{\pe}\}})} 
\to E^2(\XX_{\Sigma_f \setminus S_p(E)}(K)^{\omega \psi^{-1}})^{\iota}(1) \to 0.
\]
Moreover, the image of the first map from $\Z_p(1)^{\psi}$ is finite.
\item[(2)]
Suppose either $d \geq 2$ or $\VV \supsetneqq S_p(E)$.
Then there exist $\RR^{\psi}$-modules $\cA$ and $\BB$ which fit in exact sequences
\begin{equation}\label{eq:62}
0 \to \XX_{\ol{\UU^c}}(K)^{\psi} \to \cA \to Z_{(\Sigma_f \setminus \VV) \cup \UU^c}^0(K)^{\psi} \to 0
\end{equation}
and
\begin{equation}\label{eq:63}
0 \to \XX_{(\Sigma_f \setminus \VV) \cup \UU^c}(K)^{\omega \psi^{-1}} \to \BB \to Z_{\ol{\UU^c}}^0(K)^{\omega \psi^{-1}} \to 0
\end{equation}
such that we have an exact sequence
\[
0 \to \cA \to 
\frac{\RR^{\psi}}{(\LL^{\alg, \psi}_{\Sigma, \VV \setminus \cS_1^c}, \LL^{\alg, \psi}_{\Sigma, \VV \setminus \cS_2^c})}
\to E^2(\BB)^{\iota}(1) \to 0.
\]
\end{itemize}
\end{thm}

We will give explicit constructions of $\cA$ and $\BB$ in the proof.
Before that, we give a remark and a corollary.

\begin{rem}\label{rem:BCG+b_main}
We can deduce the main result of \cite{BCG+b} from Theorem \ref{thm:105} as follows.
The result \cite[Theorem 5.12]{BCG+b} deals with the case where $p \nmid [K: \tilde{E}]$ and $\Sigma = S_p(E) \cup S_{\infty}(E)$, and concerns the second Chern classes.
The condition \eqref{eq:134} holds trivially by $\Sigma_f = S_p(E)$.
Then we only have to recall that $c_2$ is additive with respect to exact sequences, and that $c_2(E^2(M)) = c_2(M)$ for pseudo-null module $M$ (see \cite[Remark 5.11]{BCG+b}).
\end{rem}

\begin{cor}\label{cor:73}
Suppose $l = 1$.
Let $\psi$ be a character of $\Gal(K/\ttilde{E})^{(p')}$.
Suppose \eqref{eq:134} for any $v \in \VV \setminus S_p(E)$ and that $Z_v(K)^{\psi} = 0$ for $v \in \Sigma_f \setminus \VV$.
Then the following are equivalent.
\begin{itemize}
\item[(a)]
The module
\[
\frac{\RR^{\psi}}{(\LL^{\alg, \psi}_{\Sigma, \VV \setminus \cS_1^c}, \LL^{\alg, \psi}_{\Sigma, \VV \setminus \cS_2^c})}
\]
is pseudo-null.
\item[(b)]
Both $\XX_{(\Sigma_f \setminus \VV) \cup \UU^c}(K)^{\omega \psi^{-1}}$ and $\XX_{\ol{\UU^c}}(K)^{\psi}$ are pseudo-null.
\end{itemize}
Moreover, if these conditions hold, we have 
\[
\XX_{\ol{\UU^c}}(K)^{\psi}_{\fin} = 0
\]
unless $d = 1$, $\VV = S_p(E)$, and $\psi = \omega$, in which case $\XX(K)^{\omega}_{\fin}$ is cyclic.
\end{cor}

\begin{proof}
This is a generalization of \cite[Proposition 5.10]{BCG+b}, whose proof we will follow.
Firstly, Theorem \ref{thm:105} directly shows (b) $\Rightarrow$ (a).
The final assertion on $\XX_{\ol{\UU^c}}(K)^{\psi}_{\fin}$ under (a) and (b) follows from Theorem \ref{thm:105} and an algebraic proposition (e.g., \cite[Lemma A.3]{BCG+}) that $\left(\frac{\RR^{\psi}}{(\LL^{\alg, \psi}_{\Sigma, \VV \setminus \cS_1^c}, \LL^{\alg, \psi}_{\Sigma, \VV \setminus \cS_2^c})}\right)_{\fin} = 0$.

It remains to show (a) $\Rightarrow$ (b), so let us suppose (a).
We first observe that, by Definition \ref{defn:38}, the element $\LL_{\Sigma, \VV \setminus \cS_i^c}^{\alg}$ annihilates $H^2(C_{\Sigma, \VV \setminus \cS_i^c})$ for each $i = 1, 2$.

We shall show the pseudo-nullity of $\XX_{\ol{\UU^c}}(K)^{\psi}$.
By \eqref{eq:77} and the natural surjective homomorphism
\[
\XX_{\VV \setminus \cS_i^c}(K) \twoheadrightarrow \XX_{\VV \setminus \ol{\UU}}(K),
\]
the element $\LL_{\Sigma, \VV \setminus \cS_i^c}^{\alg}$ also annihilates $\XX_{\VV \setminus \ol{\UU}}(K)$.
Therefore, the pseudo-nullity of $\XX_{\ol{\UU^c}}(K)^{\psi} = \XX_{\VV \setminus \ol{\UU}}(K)^{\psi}$ (similar as Lemma \ref{lem:135}) follows from (a).

Next we show the pseudo-nullity of $\XX_{(\Sigma_f \setminus \VV) \cup \UU^c}(K)^{\omega \psi^{-1}}$.
As in the proof of Proposition \ref{prop:46}, triangle \eqref{eq:137} induces an exact sequence
\[
0 \to D_{\TT_i}(K) \to H^2(C_{\Sigma, \VV \setminus \UU^c}) \to H^2(C_{\Sigma, \VV \setminus \cS_i^c}),
\]
where the injectivity follows from Lemma \ref{lem:37}.
We know that $E^1(H^2(C_{\Sigma, (\Sigma_f \setminus \VV) \cup \UU^c}))^{\iota}(1)$ coincides with $H^2(C_{\Sigma, \VV \setminus \UU^c})_{\tor}$ by Proposition \ref{prop:20}(1).
Moreover, $D_{\TT_i}(K)$ is torsion-free by Proposition \ref{prop:22}(2).
From these observations, we see that $E^1(H^2(C_{\Sigma, (\Sigma_f \setminus \VV) \cup \UU^c}))^{\iota}(1)$ maps injectively into $H^2(C_{\Sigma, \VV \setminus \cS_i^c})$.
In particular, the element $\LL_{\Sigma, \VV \setminus \cS_i^c}^{\alg, \psi}$ annihilates $E^1(H^2(C_{\Sigma, (\Sigma_f \setminus \VV) \cup \UU^c})^{\omega \psi^{-1}})^{\iota}(1)$.
By assumption (a), it follows that $E^1(H^2(C_{\Sigma, (\Sigma_f \setminus \VV) \cup \UU^c})^{\omega \psi^{-1}})^{\iota}(1)$ is pseudo-null.
Therefore, $H^2(C_{\Sigma, (\Sigma_f \setminus \VV) \cup \UU^c})^{\omega \psi^{-1}}$ is pseudo-null.
\end{proof}

In the rest of this section, we prove Theorem \ref{thm:105}.
We first show a couple of propositions that are valid without assuming $l = 1$.

The following proposition is a motivation for the conditions in Theorem \ref{thm:105}.
Note that the pseudo-nullity of $\XX_{(\Sigma_f \setminus \VV) \cup \UU^c}(K)^{\omega \psi^{-1}}$ is a stronger condition than Greenberg's conjecture.
In fact it is not necessarily true; in \cite{Kata_04}, the author obtained examples for which the pseudo-nullity does not hold for tamely ramified Iwasawa modules (when $E$ is imaginary quadratic, $K = \tilde{E}(\mu_p)$, and $\psi = \omega$).

\begin{prop}\label{prop:22}
The following are true.
\begin{itemize}
\item[(1)]
Let $\psi$ be a character of $\Gal(K/\ttilde{E})^{(p')}$.
The module $H^2(C_{\Sigma, \VV \setminus \UU^c})^{\psi}$ is torsion-free if and only if \eqref{eq:134} holds for any $v \in \VV \setminus S_p(E)$ and moreover $\XX_{(\Sigma_f \setminus \VV) \cup \UU^c}(K)^{\omega \psi^{-1}}$ is pseudo-null.
\item[(2)]
$D_{\pe}(K)$ is torsion-free for each $\pe \in S_p(E)$.
\end{itemize}
\end{prop}

\begin{proof}
This proposition can be proved in the same way as in \cite[Remark 4.2.5]{BCG+}, as follows.
By Proposition \ref{prop:20}, the module $H^2(C_{\Sigma, \VV \setminus \UU^c})^{\psi}$ (resp.~$D_{\pe}(K)$) is torsion-free if and only if $E^1(H^2(C_{\Sigma, (\Sigma_f \setminus \VV) \cup \UU^c}))^{\omega \psi^{-1}} = 0$ (resp.~$E^1(H^2_{\Iw}(K_{\pe}, \Z_p)) = 0$), 
that is, $H^2(C_{\Sigma, (\Sigma_f \setminus \VV) \cup \UU^c})^{\omega \psi^{-1}}$ (resp.~$H^2_{\Iw}(K_{\pe}, \Z_p)$) is pseudo-null.
Therefore, claim (1) follows from \eqref{eq:77} and $Z_v(K)_{\PN} = 0$ for finite prime $v$ outside $p$, and claim (2) follows from \eqref{eq:42} and Lemma \ref{lem:a13}.
\end{proof}

\begin{lem}\label{lem:135}
Let $\psi$ be a character of $\Gal(K/\ttilde{E})^{(p')}$.
Suppose that \eqref{eq:134} holds for any $v \in \VV \setminus S_p(E)$.
Then we have $\XX_{\VV \setminus \UU^c}(K)^{\psi} \simeq \XX_{\UU}(K)^{\psi}$.
\end{lem}

\begin{proof}
We have an exact sequence
\[
\bigoplus_{v \in \VV \setminus S_p(E)} D_v(K) 
\to \XX_{\VV \setminus \UU^c}(K)
\to \XX_{\UU}(K) \to 0.
\]
For each $v \in \VV \setminus S_p(E)$, assumption \eqref{eq:134} and the duality \eqref{eq:94} imply that $D_v(K)^{\psi} = 0$.
Hence the lemma follows.
\end{proof}

\begin{prop}\label{prop:74}
The following are true.
\begin{itemize}
\item[(1)]
Let $\psi$ be a character of $\Gal(K/\ttilde{E})^{(p')}$.
Suppose that \eqref{eq:134} holds for any $v \in \VV \setminus S_p(E)$ and moreover $\XX_{(\Sigma_f \setminus \VV) \cup \UU^c}(K)^{\omega \psi^{-1}}$ is pseudo-null.
Then we have a natural isomorphism
\[
\bigcap_{\RR^{\psi}}^l H^2(C_{\Sigma, \VV \setminus \UU^c})^{\psi} \simeq \Det_{\RR^{\psi}}(C_{\Sigma, \VV \setminus \UU^c})^{\psi}.
\]
\item[(2)]
For each $1 \leq i \leq n$, we have a natural isomorphism
\[
\bigcap_{\RR}^l D_{\TT_i}(K) \simeq \Det_{\RR}^{-1}(C_{\TT_i}^{\loc}).
\]
\end{itemize}
\end{prop}

\begin{proof}
Recall that, in Proposition \ref{prop:46}, we checked that the conditions of Proposition \ref{prop:30} hold for $C = C_{\Sigma, \VV \setminus \UU^c}^{\psi}[1]$ and for $C = C_{\TT_i}^{\loc}$.
Then by Proposition \ref{prop:22} we can apply Proposition \ref{prop:43} to deduce the proposition.
\end{proof}

\begin{prop}\label{prop:a18}
Suppose $d = 1$ (i.e., $E$ is an imaginary quadratic field).
Then we have the following.
\begin{itemize}
\item[(1)]
The $\RR$-module $\Z_p(1)$ does not appear as a submodule or a quotient module of $\XX(K)$.
\item[(2)]
If we assume that $\XX(K)^{\omega}$ is pseudo-null, then $\Z_p(1)$ does not appear as a submodule or a quotient module of $E^2(\XX(K))$.
\end{itemize}
\end{prop}

\begin{proof}
(1) follows from \cite[Proposition 4.1.15]{BCG+}.
For (2), we observe the duality
\[
\XX(K)^{\omega}_{/\fin} \simeq E^2(E^2(\XX(K)^{\omega}_{/\fin}))
\]
(see Proposition \ref{prop:88}). 
Note also that $E^2(\Z_p(1)) \simeq \Z_p(1)$.
Then it is easy to see that $\Z_p(1)$ is a sub (resp.~a quotient) of $E^2(\XX(K))^{\omega}$ if and only if $\Z_p(1)$ is a quotient (resp.~a sub) of $\XX(K)^{\omega}$.
Therefore, claim (2) follows from claim (1).
\end{proof}

\begin{proof}[Proof of Theorem \ref{thm:105}]
For a while we treat all cases simultaneously.
Let $\TT_1 = \{\pe\}$ and $\TT_2 = \{\ol{\pe}\}$.
By Propositions \ref{prop:20} and \ref{prop:22} we have a commutative diagram with exact rows
\[
\xymatrix{
	0 \ar[r]
	& D_{\{\pe, \ol{\pe}\}}(K)^{\psi} \ar[r] \ar[d]_{f_1}
	& D_{\{\pe, \ol{\pe}\}}(K)^{\psi, **} \ar[d]_{f_2} \ar[r]
	& E^2(Z_{\{\pe, \ol{\pe}\}}(K)^{\omega \psi^{-1}})^{\iota}(1) \ar[r] \ar[d]_{f_3}
	& 0 \\
	0 \ar[r]
	& H^2(C_{\Sigma, \VV \setminus \UU^c})^{\psi} \ar[r]
	& H^2(C_{\Sigma, \VV \setminus \UU^c})^{\psi, **} \ar[r]
	& E^2(H^2(C_{\Sigma, (\Sigma_f \setminus \VV) \cup \UU^c})^{\omega \psi^{-1}})^{\iota}(1) \ar[r]
	& W_{(\Sigma_f \setminus \VV) \cup \UU^c}^{\psi}.
}
\]
By Proposition \ref{prop:74}, this diagram can be identified with the $\psi$-component of that in Proposition \ref{prop:46}.
We shall show that the snake lemma applied to this diagram proves Theorem \ref{thm:105}.

By Proposition \ref{prop:47}, we know that 
\[
\Coker(f_2) \simeq \frac{\RR^{\psi}}{(\LL^{\alg, \psi}_{\Sigma, \VV \setminus \cS_1^c}, \LL^{\alg, \psi}_{\Sigma, \VV \setminus \cS_2^c})}.
\]

To study $f_1$, we use the exact sequence
\begin{equation}\label{eq:a16}
0 \to \XX_{\UU}(K)^{\psi} 
\to H^2(C_{\Sigma, \VV \setminus \UU^c})^{\psi} 
\to Z_{(\Sigma_f \setminus \VV) \cup \UU^c}^0(K)^{\psi} \to 0
\end{equation}
obtained by \eqref{eq:77} and Lemma \ref{lem:135}.
Putting $\cA = \Coker(f_1)$, we deduce an exact sequence
\begin{equation}\label{eq:a17}
0 \to \XX_{\ol{\UU^c}}(K)^{\psi} \to \cA \to Z_{(\Sigma_f \setminus \VV) \cup \UU^c}^0(K)^{\psi} \to 0,
\end{equation}
which is claimed in the theorem.

Similarly, to study $f_3$, we use the exact sequence
\begin{equation}\label{eq:61}
0 \to \XX_{(\Sigma_f \setminus \VV) \cup \UU^c}(K)^{\omega \psi^{-1}} 
\to H^2(C_{\Sigma, (\Sigma_f \setminus \VV) \cup \UU^c})^{\omega \psi^{-1}} 
\to Z_{\UU}^0(K)^{\omega \psi^{-1}} \to 0
\end{equation}
obtained by \eqref{eq:77} and \eqref{eq:134}.
All modules in \eqref{eq:61} are pseudo-null by assumption.

Now we consider (1) (so $\UU = S_p(E)$).
In this case, since we have $E^i(\Z_p) = 0$ and $E^i(Z_{S_p(E)}(K)) = 0$ for $i \neq 2$, we deduce from \eqref{eq:61} an exact sequence
{\small
\[
0 \to E^2(\Z_p)^{\omega \psi^{-1}} \to E^2(Z_{S_p(E)}(K))^{\omega \psi^{-1}} 
\to E^2(H^2(C_{\Sigma, \Sigma_f \setminus S_p(E)}))^{\omega \psi^{-1}}
 \to E^2(\XX_{\Sigma_f \setminus S_p(E)}(K))^{\omega \psi^{-1}} \to 0.
\]
}
Hence we have
\[
\Ker(f_3) \simeq \Z_p(1)^{\psi}, 
\qquad \Coker(f_3) \simeq E^2(\XX_{\Sigma_f \setminus S_p(E)}(K)^{\omega \psi^{-1}})^{\iota}(1).
\]

Putting these together, the snake lemma induces an exact sequence
\[
\Z_p(1)^{\psi} \to \cA
\to \frac{\RR^{\psi}}{(\LL^{\alg, \psi}_{\Sigma, \{\pe\}}, \LL^{\alg, \psi}_{\Sigma, \{\ol{\pe}\}})} \to E^2(\XX_{\Sigma_f \setminus S_p(E)}(K)^{\omega \psi^{-1}})^{\iota}(1) \to W_{\Sigma_f \setminus S_p(E)}^{\psi}.
\]
The last map to $W_{\Sigma_f \setminus S_p(E)}^{\psi}$ is zero because, we only have to consider the case $\Sigma_f = S_p(E)$ and then Proposition \ref{prop:a18}(2) applies.
Let us moreover show that the image of the first map from $\Z_p(1)^{\psi}$ is finite.
It is easy to see that the module $Z_{\Sigma_f \setminus S_p(E)}(K)$ does not contain $\Z_p(1)$ as a submodule.
Then Proposition \ref{prop:a18}(1) implies that $\cA$ does not contain $\Z_p(1)$ as a submodule, so the claim holds.

Now we consider (2).
We first observe that the natural map
\[
Z_{\UU}^0(K)^{\omega \psi^{-1}} \to Z_{\{\pe, \ol{\pe}\}}(K)^{\omega \psi^{-1}}
\]
 is surjective.
If $d \geq 2$, this follows from $\UU \supsetneqq \{\pe, \ol{\pe}\}$.
If $\VV \supsetneqq S_p(E)$, then condition \eqref{eq:134} clearly implies $(\Z_p)^{\omega \psi^{-1}} = 0$ (i.e., $\psi \neq \omega$).
Thus the surjectivity holds.

Therefore, by \eqref{eq:61}, we can define a module $\BB$ by the following exact sequence
\begin{equation}\label{eq:301}
0 \to \BB \to H^2(C_{\Sigma, (\Sigma_f \setminus \VV) \cup \UU^c})^{\omega \psi^{-1}} 
\to Z_{\{\pe, \ol{\pe}\}}(K)^{\omega \psi^{-1}} \to 0,
\end{equation}
and then $\BB$ fits in the short exact sequence (with $\BB$ in the middle) claimed in the theorem.
On the other hand, from \eqref{eq:301} we deduce an exact sequence
\[
0 \to E^2(Z_{\{\pe, \ol{\pe}\}}(K)^{\omega \psi^{-1}}) \to E^2(H^2(C_{\Sigma, (\Sigma_f \setminus \VV) \cup \UU^c})^{\omega \psi^{-1}}) \to E^2(\BB) \to 0.
\]
Hence we have
\[
\Ker(f_3) = 0,
\qquad \Coker(f_3) \simeq E^2(\BB)^{\iota}(1).
\]

Putting these together, the snake lemma induces an exact sequence
\[
0 \to \cA \to 
\frac{\RR^{\psi}}{(\LL^{\alg, \psi}_{\Sigma, \VV \setminus \cS_1^c}, \LL^{\alg, \psi}_{\Sigma, \VV \setminus \cS_2^c})}
\to E^2(\BB)^{\iota}(1) \to W_{(\Sigma_f \setminus \VV) \cup \UU^c}^{\psi}.
\]
Finally we show that the last map is zero.
We only have to deal with the case where $\Sigma_f = \VV$ and $\UU^c = \emptyset$ (so $d = 1$).
In that case, we have $\BB \simeq \XX(K)^{\omega \psi^{-1}}$, and Proposition \ref{prop:a18}(2) again concludes the proof.
\end{proof}

\renewcommand{\thesection}{\Alph{section}}
\setcounter{section}{0}

\section{Properties of Fitting ideals}\label{sec:87}

Let $\RR$ be a ring which contains a regular local ring $\Lambda$ as in the first paragraph of \S \ref{subsec:a22}.

Let $\PP^1_{\RR}$ be the category of finitely generated torsion $\RR$-modules $P$ with $\pd_{\RR}(P) \leq 1$.
This is equivalent to that $P$ admits an exact sequence of the form
\begin{equation}\label{eq:pres}
0 \to \RR^a \overset{H}{\to} \RR^a \to P \to 0
\end{equation}
 (with an integer $a \geq 0$).
Then the Fitting ideal $\Fitt_{\RR}(P)$ is by definition generated by the single element $\det(H)$, which is a non-zero-divisor, so in particular $\Fitt_{\RR}(P)$ is invertible as a fractional ideal.

\begin{prop}\label{prop:A1}
For each $P \in \PP^1_{\RR}$, the following hold.
\begin{itemize}
\item[(1)]
We have $E^1(P) \in \PP^1_{\RR}$ and $E^1(E^1(P)) \simeq P$.
\item[(2)]
We have $\Fitt_{\RR}(E^1(P)) = \Fitt_{\RR}(P)$.
\end{itemize}
\end{prop}

\begin{proof}
This proposition is well-known to experts, but for the convenience of the reader we give a proof.
Let us take a sequence of the form \eqref{eq:pres}, which yields an exact sequence 
\begin{equation}\label{eq:pres'}
0 \to (\RR^a)^* \overset{H^*}{\to} (\RR^a)^* \to E^1(P) \to 0,
\end{equation}
 where $H^*$ is the map induced by $H$.
This implies that $E^1(P) \in \PP^1_{\RR}$ and
\[
\Fitt_{\RR}(E^1(P)) = (\det(H^*)) = (\det(H)) = \Fitt_{\RR}(P).
\]
Moreover, \eqref{eq:pres'} yields an exact sequence 
\[
0 \to (\RR^a)^{**} \overset{H^{**}}{\to} (\RR^a)^{**} \to E^1(E^1(P)) \to 0.
\]
Comparing this sequence with \eqref{eq:pres}, together with $(\RR^a)^{**} \simeq \RR^a$, shows $E^1(E^1(P)) \simeq P$.
\end{proof}

Now we consider codimension two analogues of Proposition \ref{prop:A1}.
Let $\PP^2_{\RR}$ be the category of pseudo-null $\RR$-modules $M$ with $\pd_{\RR}(M) \leq 2$.

\begin{prop}\label{prop:88}
For each $M \in \PP^2_{\RR}$, the following hold.
\begin{itemize}
\item[(1)]
We have $E^2(M) \in \PP^2_{\RR}$ and $E^2(E^2(M)) \simeq M$.
\item[(2)]
We have $\Fitt_{\RR}(E^2(M)) = \Fitt_{\RR}(M)$.
\end{itemize}
\end{prop}

First we show claim (1).

\begin{proof}[Proof of Proposition \ref{prop:88}(1)]
This is proved in a similar way as Proposition \ref{prop:A1}(1).
Let us choose a module $P \in \PP^1_{\RR}$ and a surjective homomorphism from $P$ to $M$; for instance, if $f \in \RR$ is an annihilator of $M$ that is a non-zero-divisor, we may take $P$ as a direct sum of copies of $\RR/(f)$.
Then, letting $Q$ denote the kernel of $P \to M$, we obtain an exact sequence
\begin{equation}\label{eq:pres2}
0 \to Q \to P \to M \to 0.
\end{equation}
By $M \in \PP^2_{\RR}$ and $P \in \PP^1_{\RR}$, we have $Q \in \PP^1_{\RR}$.
Then sequence \eqref{eq:pres2} induces an exact sequence
\begin{equation}\label{eq:pres2'}
0 \to E^1(P) \to E^1(Q) \to E^2(M) \to 0.
\end{equation}
This sequence, together with the fact $E^1(P), E^1(Q) \in \PP^1_{\RR}$ by Proposition \ref{prop:A1}(1), implies that $\pd_{\RR}(E^2(M)) \leq 2$.
Since $M$ is pseudo-null, so is $E^2(M)$.
Therefore, we have $E^2(M) \in \PP^2_{\RR}$.

Moreover, comparing \eqref{eq:pres2} with the sequence 
\[
0 \to E^1(E^1(Q)) \to E^1(E^1(P)) \to E^2(E^2(M)) \to 0
\]
 induced by \eqref{eq:pres2'}, and using the facts $E^1(E^1(P)) \simeq P$ and $E^1(E^1(Q)) \simeq Q$ by Proposition \ref{prop:A1}(1), we obtain $E^2(E^2(M)) \simeq M$.
\end{proof}

For the proof of Proposition \ref{prop:88}(2), we need an auxiliary lemma.

\begin{lem}\label{lem:a88}
For any finitely generated torsion $\RR$-module $M$, there exists a finite presentation of $M$
\[
\RR^a \overset{H}{\to} \RR^b \to M \to 0
\]
such that all $b \times b$ minors of the presentation matrix of $H$ are non-zero-divisors of $\RR$.
\end{lem}

\begin{proof}
If $\Lambda$ has only finitely many elements, then $\Lambda$ is a finite field and the only torsion module is the zero module, so the assertion is trivial.
Therefore, we may assume that $\Lambda$ has infinitely many elements.

Let $\RR^m \overset{H}{\to} \RR^n \to M \to 0$ be any finite presentation of $M$ ($m \geq n$).
We identify $H$ with its presentation matrix in $M_{m, n}(\RR)$.
Let us fix a non-zero-divisor $f \in \RR$ that annihilates $M$.
For an element $X = (x_{i j})_{i, j} \in M_{m, n}(\Lambda)$, we consider a matrix
\[
H_X = \begin{pmatrix}
f I_n \\ H + f X
\end{pmatrix} \in M_{m+n, n}(\RR),
\]
where $I_n$ denotes the identity matrix of size $n$.
Then $H_X$ can also be regarded as a finite presentation of $M$.
We will find an element $X \in M_{m, n}(\Lambda)$ such that all $n \times n$ minors of $H_X$ are non-zero-divisors, which would complete the proof of the lemma.

By a pair $(I, J)$, we will mean a pair of subsets $I \subset \{1, 2, \dots, m\}$ and $J \subset \{1, 2, \dots, n\}$ such that $\# I = \# J$.
For a matrix $A = (a_{ij})_{1 \leq i \leq m, 1 \leq j \leq n}$ of size $m \times n$, let us write
\[
A_{I, J} = (a_{ij})_{i \in I, j \in J},
\]
i.e., the square submartix of $A$ obtained by picking the rows in $I$ and the columns in $J$.

We observe that any $n \times n$ minor of $H_X$ is the product of a power of $f$ and a minor of $H + fX$ (not necessary of degree $n$).
Therefore, the required property of $X$ is equivalent to that $\det((H + fX)_{I, J}) \in \RR$ is a non-zero-divisor for any pair $(I, J)$.

Let us write $\RR[X]$ (resp.~$\Lambda[X]$) for the polynomial ring over $\RR$ (resp.~$\Lambda$) in variables $\{x_{ij}\}_{1 \leq i \leq m, 1 \leq j \leq n}$; here we regard $\{x_{ij}\}_{i, j}$ as indeterminates.
For example, $\det(X_{I, J}) \in \Lambda[X]$ is a nonzero homogenous polynomial of degree $\# I (= \# J)$.
By the definition of the determinant, we have
\begin{equation}\label{eq:Adet2}
\det((H + fX)_{I, J}) = f^{\# I} \det(X_{I, J}) + (\text{lower degree}),
\end{equation}
where (lower degree) denotes a polynomial in $\RR[X]$ whose degree is strictly less than $\# I$.
We put
\[
D_X = \prod_{(I, J)} \det((H + fX)_{I, J}),
\]
where $(I, J)$ runs over all the pairs we are considering.
By taking the product of \eqref{eq:Adet2}, we obtain
\begin{equation}\label{eq:Adet}
D_X = f^N F(X) + (\text{lower degree})
\end{equation}
with $N = \sum_{(I, J)} \# I$, $F(X) = \prod_{(I, J)} \det(X_{I, J})$, and (lower degree) denotes a polynomial in $\RR[X]$ whose degree is strictly less than that of $F(X)$.

Since $\RR$ is free of finite rank over $\Lambda$, we have the norm map $\sN: \RR \to \Lambda$; for $a \in \RR$, the norm $\sN(a)$ is defined as the determinant of the presentation matrix of the $\Lambda$-homomorphism $a: \RR \to \RR$ (with respect to any basis).
Note that an element $a \in \RR$ is a non-zero-divisor if and only if $\sN(a) \neq 0$.
By \eqref{eq:Adet}, we find
\[
\sN(D_X) = \sN(f)^N F(X)^{\rank_{\Lambda}(\RR)} + (\text{lower degree}),
\]
where (lower degree) denotes a polynomial in $\Lambda[X]$ whose degree is strictly less than that of $F(X)^{\rank_{\Lambda}(\RR)}$.

Now it is enough to show that $\sN(D_X) \neq 0$ for some $X \in M_{m, n}(\Lambda)$.
Since $\det(X_{I, J})$ is a nonzero homogenous polynomial in $\Lambda[X]$, so is $F(X)$.
Recall that we assume $\Lambda$ has infinitely many elements.
Therefore, there exists an element $X_0 \in M_{m, n}(\Lambda)$ such that $F(X_0) \neq 0$.
Once we fix such an $X_0$, we may regard $\sN(D_{\lambda X_0})$ as a polynomial in $\lambda \in \Lambda$ whose leading coefficient is $\sN(f)^N F(X_0)^{\rank_{\Lambda}(\RR)} \neq 0$.
Then we find an element $\lambda \in \Lambda$ such that $\sN(D_{\lambda X_0}) \neq 0$.
This completes the proof of the lemma.
\end{proof}

\begin{proof}[Proof of Proposition \ref{prop:88}(2)]
By claim (1), it is enough to show a single inclusion, say $\supset$.
Let us take a finite presentation $\RR^a \overset{H}{\to} \RR^b \overset{\pi}{\to} M \to 0$ as in Lemma \ref{lem:a88}.
Take any $b \times b$ submatrix of $H$ and let $\alpha: \RR^b \to \RR^b$ be the corresponding homomorphism.
By the definition of Fitting ideals, it is enough to show $\det(\alpha) \in \Fitt_{\RR}(E^2(M))$.

Put $P = \Coker(\alpha)$, so we have $\Fitt_{\RR}(P) = (\det(\alpha))$.
Since $\alpha$ is injective by the choice of $H$, we also have $P \in \PP^1_{\RR}$.
On the other hand, we have $\pi \circ \alpha = 0$, so $\pi$ induces a surjective homomorphism $P \to M$.
We define $Q$ as its kernel, so we obtain an exact sequence
\begin{equation}\label{eq:90}
0 \to Q \to P \to M \to 0.
\end{equation}
Observe that $M \in \PP^2_{\RR}$ and $P \in \PP^1_{\RR}$ imply $Q \in \PP^1_{\RR}$.
Then \eqref{eq:90} induces an exact sequence
\[
0 \to E^1(P) \to E^1(Q) \to E^2(M) \to 0.
\]
In particular, since $E^2(M)$ is a quotient of $E^1(Q)$, we have 
\[
\Fitt_{\RR}(E^2(M)) \supset \Fitt_{\RR}(E^1(Q)).
\]
Since $Q \in \PP^1_{\RR}$, Proposition \ref{prop:A1}(2) implies $\Fitt_{\RR}(E^1(Q)) = \Fitt_{\RR}(Q)$.
Finally we show $\Fitt_{\RR}(Q) = \Fitt_{\RR}(P)$.
Let $\qu$ be any height one prime of $\Lambda$, and recall that the subscript $(-)_{\qu}$ denotes the localization with respect to the multiplicative set $\Lambda \setminus \qu$.
By \eqref{eq:90} and the pseudo-nullity of $M$, we have $Q_{\qu} \simeq P_{\qu}$, so 
$\Fitt_{\RR}(Q)\RR_{\qu} = \Fitt_{\RR}(P)\RR_{\qu}$.
Then the claim $\Fitt_{\RR}(Q) = \Fitt_{\RR}(P)$ follows from this, since both sides are invertible (moreover principal) fractional ideals of $\RR$.
\end{proof}

\begin{rem}\label{rem:95}
As mentioned in Remark \ref{rem:BCG+b_main}, when $\RR$ is a product of regular local rings, we have $c_2(E^2(M)) = c_2(M)$ for a pseudo-null module $M$.
Although the statement of Proposition \ref{prop:88} is quite similar, the proofs have nothing in common.
\end{rem}

Let us suppose that $\RR$ is a product of regular local rings.
Then, for each finitely generated torsion $\RR$-module $M$, we have a classical definition of characteristic ideal $\cha_{\RR}(M)$, using the structure theorem of finitely generated modules (up to pseudo-null modules).
By definition $\cha_{\RR}(M)$ is principal.
The characteristic ideals satisfy the additivity properties with respect to exact sequences, and we have $\cha_{\RR}(P) = \Fitt_{\RR}(P)$ for $P \in \PP^1_{\RR}$.
In fact, these properties characterize $\cha_{\RR}(-)$.

The following proposition is used in \S \ref{subsec:132}.

\begin{prop}\label{prop:91}
Suppose that $\dim(\Lambda) = 2$ and that $\RR$ is a product of regular local rings.
For each finitely generated torsion $\RR$-module $M$, we have
\[
\Fitt_{\RR}(M) = \cha_{\RR}(M) \Fitt_{\RR}(E^2(M)).
\]
\end{prop}

\begin{proof}
Consider the exact sequence $0 \to  M_{\PN} \to M \to M_{/\PN} \to 0$.
We have $M_{/\PN} \in \PP^1_{\RR}$ by applying the Auslander-Buchsbaum formula.
Then a well-known property of Fitting ideals (see, e.g., \cite[Lemma 3]{CG98}) tells us
\[
\Fitt_{\RR}(M) = \Fitt_{\RR}(M_{/\PN}) \Fitt_{\RR}(M_{\PN}).
\]
By basic properties of characteristic ideals, we have
\[
\Fitt_{\RR}(M_{/\PN}) = \cha_{\RR}(M_{/\PN}) = \cha_{\RR}(M).
\]
On the other hand, we have $M_{\PN} \in \PP^2_{\RR}$, so Proposition \ref{prop:88} implies
\[
\Fitt_{\RR}(M_{\PN}) = \Fitt_{\RR}(E^2(M_{\PN})) = \Fitt_{\RR}(E^2(M)).
\]
This completes the proof.
\end{proof}

\section*{Acknowledgments}

I would like to thank Masato Kurihara for encouraging me in this research and giving valuable comments.
I am also grateful to Mahiro Atsuta and Ryotaro Sakamoto for discussion on the papers \cite{BCG+} and \cite{BCG+b}.
This research was supported by JSPS KAKENHI Grant Number 19J00763.

{
\bibliographystyle{abbrv}
\bibliography{biblio}
}

\end{document}